\documentclass[10pt,a4paper]{article}
\usepackage[T1]{fontenc}
\usepackage[utf8]{inputenc}
\pdfoutput=1
\usepackage{lmodern}
\usepackage{textcomp}
\usepackage{microtype}
\usepackage[DIV=10,BCOR=5mm,headinclude=false,footinclude=false]{typearea}
\usepackage[font=small,labelfont=bf]{caption}
\usepackage{titlesec}
\usepackage{tocloft}
\usepackage{amsmath,amssymb}
\usepackage{enumerate} 
\usepackage{amsfonts}
\usepackage{graphicx}
\usepackage{amsthm}
\usepackage{yfonts}
\usepackage{tikz-cd}
\usepackage{amssymb}
\usepackage{mathrsfs}
\usepackage[parfill]{parskip}
\usepackage{hyperref}
\usepackage[nameinlink]{cleveref}
\usepackage{MnSymbol }
\usepackage{verbatim}
\usepackage{array}
\usepackage{mathdots}

\newtheorem{theorem}{Theorem}
\newtheorem{conjecture}{Conjecture}
\newtheorem{corollary}{Corollary}[theorem]
\newtheorem{lemma}{Lemma}[section]
\newtheorem{prop}{Proposition}[section]

\theoremstyle{remark}
\newtheorem{rem}{Remark}[section]
\newcommand{\Ms}{\mathcal{M}\mathrm{ult}}
 \newcommand{\C}{\mathbb{C}}
\newcommand{\Ff}{\mathbf{F}}

\newcommand{\gl}{\mathrm{GL}}
\newcommand{\rep}{\mathfrak{Rep}}
\newcommand{\LL}{\Lambda}
\newcommand{\abs}{\lvert-\lvert}
\newcommand{\Z}{\mathrm{Z}}
\newcommand{\cus}{\mathfrak{Cus}}
\newcommand{\NN}{\mathbb{N}}
\newcommand{\fm}{\mathfrak{m}}
\newcommand{\fn}{\mathfrak{n}}
\newcommand{\De}{\Delta}
\newcommand{\Ga}{\Gamma}
\newcommand{\sh}[1]{\overset{\leftarrow}{#1}}
\newcommand{\rsa}{\rightsquigarrow}
\newcommand{\fd}{\mathfrak{d}}
\newcommand{\fdp}{\mathfrak{d}}

\newcommand{\ra}{\rightarrow}

\newcommand{\ZZ}{\mathbb{Z}}
\newcommand{\Irr}{\mathfrak{Irr}}
\newcommand{\soc}{\mathrm{soc}}
\newcommand{\cosoc}{\mathrm{cosoc}}

\newcommand{\ho}{\mathrm{Hom}}

\newcommand{\ain}[3]{#1\in \{#2,\ldots,#3\}}
\newcommand{\Dde}{\mathcal{D}_{\fa}}

\newcommand{\hra}{\hookrightarrow}
\newcommand{\irs}{\Irr^\square}
\newcommand{\Se}{\mathcal{S}eg}
\newcommand{\KK}{\mathrm{K}}
\newcommand{\irr}{\Irr}
\newcommand{\cusp}{\mathrm{cusp}}
\newcommand\restr[2]{{\left.\kern-\nulldelimiterspace #1\vphantom{\big|}\right|_{#2}}}
\newcommand{\did}{\mathbf{d}}
\newcommand{\die}{\mathbf{e}}
\newcommand{\mq}{\Ms_Q}
\newcommand{\gdim}{\mathrm{grdim}}
\newcommand{\oo}{\mathbb{O}}
\newcommand{\co}{\mathbf{Comp}}
\newcommand{\ex}{\mathrm{ext}_\Pi^1}
\newcommand{\hp}{\mathrm{hom}_\Pi}
\newcommand{\Ex}{\mathrm{Ext}_\Pi^1}
\newcommand{\Hp}{\mathrm{Hom}_\Pi}
\newcommand{\fa}{\mathfrak{a}}

\newcommand{\id}{\mathrm{Ind}}
\newcommand{\DD}{\mathbf{D}}
\renewcommand{\L}{\mathrm{L}}
\newcommand{\nur}{\nu_\rho}
\title{Poles of intertwining operators in terms of irreducible components of Lusztig's characteristic variety in type $A$}
\author{Johannes Droschl}
\begin{document}
\maketitle
\begin{abstract}
    In this paper, we propose a conjectural formula for the order of the poles of intertwining operators in the context of the representation theory of general linear groups over $p$-adic fields. 
    More specifically, we conjecturally relate the order of the pole to the dimension of a Hom-space associated with irreducible components of Lusztig's characteristic variety in type $A$. We verify the conjecture in a wide range of cases.
\end{abstract}
\section{Introduction}
The connection between the representation theory of general linear groups over $p$-adic fields and the geometry of quiver varieties of type $A_n$ has always played a crucial role in the development of the local Langlands program, \emph{cf.} \cite{Bor76} \cite{KazLus87}, \cite{ChrGinz}. However, in recent years, it has become evident that not only the quiver varieties themselves, but also their cotangent bundles and, in particular, certain Lagrangian subvarieties within them, govern many representation-theoretic questions that extend beyond the scope of the traditional Deligne-Langlands correspondence, \emph{cf.} \cite{Lus90}, \cite{Lus91}. So far this connection is only subject to speculation, and we are just able to see its shadows in terms of a large array of conjectures linking the representation theory with the geometry.
The goal of this paper is to contribute another conjecture to this growing body of work and provide supporting evidence for its validity.

To make the above more concrete, let us introduce some notation. Fix a local, non-archimedean field $\Ff$  with absolute value $\abs$. 
In this introduction, we will consider irreducible smooth representations of $\gl_n(\Ff)$ in the principal block and we denote their set by $\irr_n^0$. Throughout the main part of the paper, we will drop this assumption and work with arbitrary irreducible smooth representations of the $\Ff$-points of inner forms of $\gl_n$. We restrict ourselves right now only for the sake of simplicity.
Recall that the irreducible representations in $\irr_n^0$ are precisely those whose cuspidal support consists of characters of $\Ff^\times$ of the form $\abs^n,\,n\in \ZZ$. We are thus able to associate to $\pi\in \irr_n^0$ a vector $\did^\pi\in\NN^\ZZ$ such that $\did_n^\pi$ records the multiplicity of $\abs^n$ in the cuspidal support of $\pi$. In particular, $\did$ has only finitely many non-zero entries.
For such fixed $\did$, we let $\irr^{\did}=\{\pi\in\irr_n^0:\did^\pi=\did\}$.
On the other side, we consider for $\did$ as above the space $E^+(\did)$ of complex representations of the quiver $A_\infty^+$
\[\begin{tikzcd}
    \cdots\arrow[r]&\bullet\arrow[r]&\bullet\arrow[r]&\cdots 
\end{tikzcd}\]
with dimension vector $\did$, on which the group $G_\did\coloneq \bigtimes_{i\in\ZZ}\gl_{\did_i}(\C)$ acts via conjugation.
The cotangent bundle $T^*E^+(\did)$ admits a moment map to the Lie algebra of $G_\did$. The inverse image of $0$ under this moment map is called Lusztig's characteristic variety and denoted by $\Lambda(\did)$. Its elements allow the following interpretation. Let $\overline{\Pi}$ be the path algebra of the doubled quiver
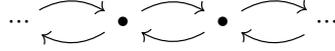
\[\begin{tikzcd}
    \cdots\arrow[r, bend left]&\arrow[l, bend left]\bullet\arrow[r, bend left]&\arrow[l, bend left]\bullet\arrow[r, bend left]&\arrow[l, bend left]\cdots 
\end{tikzcd}\]
and $\Pi$ be the quotient of $\overline{\Pi}$ by the relation $f_{i+1}e_i-e_{i-1}f_i,$ where $e_i$ is the arrow $i\rightarrow i+1$ and $f_i$ the arrow $i-1\leftarrow i$. Then $\Lambda(\did)$ consists of modules of $\Pi$ with dimension vector $\did$.

There exist bijections
\[\irr^\did\leftrightarrow\{G_\did\text{- orbits on }E^+(\did)\}\leftrightarrow \co(\did)\coloneq \{\text{irreducible components of }\Lambda(\did)\},\]
where the first arrow is explicated through Bernstein's and Zelevinsky's work \cite{Zel}, \cite{BerZel76} and the multisegments parametrizing both objects.
The second is given by sending an orbit to the closure of its conormal bundle. For later use we write $\co=\bigcup_{\did}\co(\did)$, where the union is over all dimension vectors, and recall the Aubert-Zelevinsky involution $(-)^*\colon\co\ra\co$, see \cite{Aub}, \cite{MoeWal}, \cite{ZelIII}.

The first bijection links the representation theory on the left side to the geometry on the right side via the so called Deligne-Langlands correspondence. In fact, the behavior of $\Irr^\did$ is captured by the derived category of $G_\did$-equivariant perverse sheaves on $E^+(\did)$ and the counterpart of an irreducible representation is the intersection cohomology complex supported on the closure of the orbit corresponding to it under the above bijection.
Most famously, this correspondence was used to prove the $p$-adic analog of the Kazhdan-Lusztig conjecture, see \cite{ZelIII}, \cite{ChrGinz}.

In this paper, we will focus on the bijection between irreducible representations and irreducible components, which is far less understood. 
It will be denoted by $\pi\mapsto C(\pi)$ and its inverse by $C\mapsto \pi(C)$, where $\pi\in \Irr^{\did},\,C\in\co(\did)$. One of the first appearances of Lustzig's variety in the representation theory of $\gl_n(\Ff)$ was in the form of the crystal graph of Kashiwara and Saito, \emph{cf.} \cite{KashSai}, which is closely connected to the $\rho$-derivatives of \cite{Mder}, \cite{Jader}.

For $\pi_1\in\irr_n^0,\pi_2\in\irr_m^0$ we denote by $\pi_1\times\pi_2$ the normalized parabolic induction of $\pi_1\otimes \pi_2$ along the suitable parabolic subgroup containing the upper-triangular matrices, which is a smooth representation of finite length of $\gl_{n+m}(\Ff)$. We call $\pi\in\irr_n^0$ $\square$-irreducible if $\pi\times\pi$ is irreducible.
This notation is motivated by similar constructions in the setting of quantum and cluster algebras, where the corresponding modules are called \emph{real}, see for example in the work of Geiss, Herandez, Leclerc, Schröer, Kang, Kashiwara, Kim, Oh, and many others, \emph{cf.} \cite{HerLec}, \cite{GeisLecSchr}, \cite{KKKO}. Such $\square$-irreducible representations allow one to extend the notions of crystals and $\rho$-derivatives, \emph{cf.} \cite{KKKOhead}, \cite{LapMin25}.
Historically, these representations were the first for which a geometric property in $\Lambda(\did)$ mirroring their behavior was conjectured.
\begin{conjecture}[{cf. \cite[Conjecture 5.3]{GeiSchr}, \cite[Conjecture 3.4.C]{LapMin25}}]\label{C:rigid}
        Let $C\in \co(\did)$. Then $\pi(C)\times \pi(C)$ is irreducible if and only if $C$ contains an open $G_\did$-orbit. 
\end{conjecture}
Irreducible components satisfying the condition of the conjecture are called \emph{rigid}.
More generally, the irreducibility of parabolic induction is conjecturally governed by the following criterion.
\begin{conjecture}[{cf. \cite[Conjecture 5.3]{GeiSchr}, \cite[Conjecture 3.4.A]{LapMin25}}]
    Let $C,D\in \co$. Then $\pi(C)\times\pi(D)$ is irreducible if and only if there exists open subsets $U_C\subseteq C,\, U_D\subseteq D$ such that for all $(x,y)\in U_C\times U_D$ we have 
    \[\dim_\C\Ex(x,y)=0.\]
\end{conjecture}
To go beyond irreducibility, we recall the binary operation 
\[*\colon \co(\did)\times\co(\die)\ra\co(\did+\die)\]
introduced by Aizenbud and Lapid in \cite{AizLap} given by the generic extension.
In \cite{LapMin25}, M{\'i}nguez and Lapid conjecture the following.
\begin{conjecture}[{\cite[Conjecture 5.1]{LapMin25}}]
    Let $C,D\in \co$. Then $\pi(C*D)$ is a subrepresentation of $\pi(C)\times \pi(D)$.
\end{conjecture}
These three conjectures have been proven in \cite{MinLa18} respectively \cite{LapMin25} for a large class of $C,D\in \co.$ To be more precise, the conjectures are proven if either $C$ or $D$ is \emph{balanced}, a combinatorial criterion which we recall in \Cref{S:regular}. For $C,D\in\co,\,i\in \{0,1\}$ we set \[\mathrm{ext}_\Pi^i(C,D)=\min_{x\in C,y\in D}\dim_\C\mathrm{Ext}_\Pi^i(x,y),\]
and note that the set of $x,y$ attaining this minimum is open.

In this paper, we conjecture that one can also recover certain aspects related to intertwining operators from the geometry of $\Lambda(\did)$. 
Let $\pi,\sigma\in\irr^0$ and $s\in \C$. We recall the intertwining operators \[M_{\pi,\sigma}(s)\colon \sigma\abs^s\times\pi\abs^{-s}\ra\pi\abs^{-s}\times\sigma\abs^s\]
defined by the usual integral which converges for $\mathrm{Re}(s)>>0$. One can continue $M_{\pi,\sigma}(s)$ meromorphically to the whole complex plane and we denote by $\LL(\pi,\sigma)$ the order of the pole of the intertwining operator
$M_{\pi,\sigma}(s)$ at $s=0$. Similarly, we let $\alpha(\pi,\sigma)$ the order of zero of
\[M_{\sigma,\pi}(s)\circ M_{\pi,\sigma}(s)\] at $s=0$ and set \[\fd(\pi,\sigma)\coloneq\Lambda(\pi,\sigma)+\Lambda(\sigma,\pi)+\alpha(\pi,\sigma).\]
The order of the pole of the intertwining operator, or its normalized version, see \Cref{S:norint}, has been computed in many special cases. For example, the case where both $\pi$ and $\sigma$ are certain local components of global discrete series automorphic representations has been dealt with by M{\oe}glin and Waldspurger in \cite[§ I.6]{MoeWal89}. In the setting of quantum-algebras intertwining operators have an analogues in so-called $R$-matrices, for whom the order of their poles has been computed in special cases for example by Fujita and Murakami in \cite{FujKot}, by Fujita in \cite{FujKot}, by Oh and Scrimshaw in \cite{OhScr}, and  by Kashiwara, Misra, Okado, and Yamada in \cite{KasMisOkYa}.

We present now the main conjecture of the paper.
\begin{conjecture}
    Let $C,D\in \co$. Then
    \[\Lambda(\pi(C),\pi(D))=\hp(D,C),\, \fd(\pi(C),\pi(D))=\ex(D,C).\]
\end{conjecture}
For $\did$ and $\die$ dimension vectors we denote
\[(\did,\die)\coloneq\sum_{i\in \ZZ}2\did_i\die_i-\did_i\die_{i+1}-\did_{i}\die_{i-1}.\]
\begin{lemma}[\emph{cf.} \Cref{S:AZ}, \Cref{S:AZ}]
    Let $C\in \co(\did),D\in\co(\die)$. Then we have 
    \begin{enumerate}
        \item $\alpha(\pi(C),\pi(D))=-(\did,\die),\, \Lambda(\pi(C),\pi(D))=\Lambda(\pi(D^*),\pi(C^*)),$
        \item $\hp(C,D)=\hp(D^*,C^*).$
    \end{enumerate}
\end{lemma}
The identity 
$\fd(\pi,\sigma)=\Lambda(\pi,\sigma)+\Lambda(\sigma,\pi)+\alpha(\pi,\sigma)$ is thus nothing but a (conjectural) reflection of the Crawley-Boevey-identity, see \cite[Lemma 1]{Cra},
\[\ex(C,D)=\hp(C,D)+\hp(D,C)-(\did,\die),\]
where $C\in \co(\did),D\in\co(\die)$.

The main theorem of this paper is the following.
\begin{theorem}
    Let $C,D\in \co$ such that at least one of $C,D$ is balanced. Then 
    \[\Lambda(\pi(C),\pi(D))=\hp(D,C),\, \fd(\pi(C),\pi(D))=\ex(D,C).\]
\end{theorem}
Our proof builds on the techniques developed in \cite{LapMin25}, in particular their use of \emph{saturated} and \emph{reduced} representations and modules.
We also discuss some straightforward consequences of the conjecture and verify it also in the case where $C=D$ corresponds to Leclerc's example, \emph{cf.} \cite{Lec03}, which is the smallest example of non-rigid component $C$ such that $\ex(C,C)=0$. 

In \Cref{S:norint} we quickly show how one can rephrase the conjecture (and the theorem) for intertwining operators normalized by Rankin-Selberg factors, \emph{cf.} \cite{JacPiaSha}, see also \cite{MoeWal89} and \cite{Sha83}. In this case, the order is closely linked to the behavior of Coxeter functors in this setting. We also give a fast algorithm on how to compute the order of these poles in certain cases via \emph{best matching functions}, \emph{cf.} \cite[Remark 2.2]{LapMin25}. In particular, in \Cref{C:Speh} we can give an explicit formula for the order of the poles of the intertwining operator of two so-called \emph{Speh}-representations, which play an important role in the classification of the unitary dual of $\gl_n(\Ff)$ and classical groups.
\subsection*{Acknowledgements}
I am grateful to Keyu Wang and Jianrong Li for explaining to me how to compute poles of $R$-matrices appearing in the study of quantum algebras in certain cases 
as well as to Alberto M{\'i}nguez for his continuing support and encouragement. Finally, I want to thank Erez Lapid for providing valuable feedback to earlier versions of this paper.
This work has been supported by the projects PAT4628923 and PAT4832423 of the Austrian Science Fund (FWF).
\counterwithin{theorem}{section}
\counterwithin{conjecture}{section}
\section{Homological algebra}\label{S:homalg}
In this section we review briefly several aspects of \cite{AizLap} and \cite{LapMin25}.

Let $Q^\pm$ be the quiver with vertices $I=\ZZ$ and arrows from $i\ra i\pm1$. We denote  by
\[\gdim(Q)\coloneq\{\did\in \NN^\ZZ:\text{almost all entries of }\did\text{ are }0\}.\]
A representation of $Q^\pm$ is a graded, finite dimensional $\C$-vector space $V=\bigoplus_{i\in \ZZ}V_i$ and a map $T=(T_i\colon V_i\ra V_{i\pm 1})_{i\in \ZZ}\colon V\ra V$ of degree $\pm 1$.
We call \[\gdim(V,T)=(\ldots,\dim_\C V_i,\dim_\C V_{i+1},\ldots)\] the dimension vector of $(V,T)$ and given $\did\in \gdim(Q)$, we let $E^\pm(\did)$ be the spaces of representations of $Q^\pm$ with dimension vector $\did$, on which
\[G_\did\coloneq \bigtimes_{i\in \ZZ}\gl_{\did_i}(\C)\] acts via conjugation. 
The sets of orbits are then denoted by $\oo_\pm(\did)$.

We define
\[\langle-,-\rangle_+,\, \langle-,-\rangle_-,\, (-,-)\colon \gdim(Q)\times\gdim(Q)\ra \ZZ\] via \[\langle\did,\die\rangle_+=\sum_{i\in \ZZ}\did_i\die_i-\did_{i}\die_{i+1},\, \langle\did,\die\rangle_-=\langle\die,\did\rangle_+\]
and \[(-,-)=\langle-,-\rangle_++\langle-,-\rangle_-.\]
We note that for $x,y$ representations of $Q^\pm$,
\[\dim_\C\mathrm{Hom}_{Q^\pm}(x,y)-\dim_\C\mathrm{Ext}_{Q^\pm}^1(x,y)=(\gdim(x),\gdim(y))_\pm,\] see for example \cite[Eq. 2.3]{AizLap}.
For later use we denote by \[\alpha_+(x,y)\coloneq \dim_\C\mathrm{Hom}_{Q^+}(x,y)-\dim_\C\mathrm{Ext}_{Q^+}^1(y,x)\] and hence 
\[\alpha_+(x,y)+\alpha_+(y,x)=(\gdim(x),\gdim(y)).\]
We let $\mq$ be the set of formal finite sums of segments $[a,b],\, a\le b\in \ZZ$ and define $\mu_+([a,b]) $ as the representation of $Q^+$ with vector space $V$ being of dimension $1$ in degrees $a,\ldots,b$ and $0$ otherwise. The map $T$ is given by a non-zero scalar in degrees $a,\ldots,b-1$ and the $0$-map otherwise. 
Similarly, we define $\mu_-([a,b]) $ with the same underlying vector space and $T$ as a non-zero scalar in degrees $a+1,\ldots,b$ and the $0$-map otherwise.
For $\fm=\De_1+\ldots+\De_k\in\mq$, we define
\[\mu_\pm(\fm)=\bigoplus_i\mu_\pm(\De_i).\]We let $\gdim(\fm)=\gdim(\mu_+(\fm))$ and for $\did\in \gdim(Q)$, we let \[\mq(\did)\coloneq\{\fm\in \mq:\gdim(\fm)=\did\}.\] Then there exist bijections, \emph{cf.} \cite[§12]{Lus91},
\[\Z_\pm\colon \mq(\did)\ra \oo_\pm(\did),\, \fm\mapsto G_\did\cdot\mu_\pm(\fm).\]
For a dimension vector $\did$ we let
\[\Lambda(\did)\coloneq\{(S,T)\in E^+(\did)\times E^-(\did):T_{i+1}S_i=S_{i-1}T_i\text{ for all }i\in \ZZ\},\, \Lambda\coloneq\bigcup_{\did\in\gdim(Q)}\Lambda(\did). \] be the corresponding characteristic variety introduced by Lusztig. Note that $\Lambda(\did)$ is equipped with a natural action of $G_\did$. We also recall the doubled quiver $\overline{Q}$, whose vertices are again indexed by $\ZZ$ but with arrows  $i\ra i\pm 1$ for all $i\in \ZZ$.
We let $\Pi$ be the quotient of the path-algebra of $\overline{Q}$ by the relation $f_{i+1}e_i=e_{i-1}f_i$, where $e_i$ denotes the path $i\ra i+1$ and $f_i$ the path $i\ra i-1$. The space $\Lambda$ is then nothing but the space of $\Pi$-modules, and $\Lambda(\did)$ the space of $\Pi$-modules with dimension vector $\did$. 
Let us note that for all $x\in \Lambda(\did),\, y\in\Lambda(\die) $ we have the Crawley-Boevey-identity, see \cite[Lemma 1]{Cra},
\[\dim_\C\Ex(x,y)=\dim_\C\Hp(x,y)+\dim_\C\Hp(y,x)-(\did,\die).\]
Let $\co(\did)$ denote the set of irreducible components of $\Lambda(\did)$ and \[\co\coloneq \bigcup_{\did\in\gdim(Q)}\co(\did).\]

For $C_1,C_2\in \co$, we define
$\mathcal{E}(C_1,C_2)$ as the set of $x\in \Lambda$ such that there exist $x_1\in C_1,x_2\in C_2$, a short exact sequence \[0\ra x_2\ra x\ra x_1\ra 0\] and 
the dimension of \[\Ex(x_1,x_2)\] is minimal.
\begin{theorem}[{\cite[Theorem 3.1]{AizLap}}]
The map $*\colon \co\times\co\ra \co$,
    \[C_1*C_2\coloneq\overline{\mathcal{E}(C_1,C_2)},\]
     is well defined.
\end{theorem}
The components $C_1,C_2\in \co$ are said to commute strongly if $C_1*C_2=C_1\oplus C_2$.
We define for $C_1,C_2\in \co$
\[\hp(C_1,C_2)=\min_{x_1\in C_1,x_2\in C_2}\dim_\C\Hp(x_1,x_2),\]\[ \ex(C_1,C_2)=\min_{x_1\in C_1,x_2\in C_2}\dim_\C\Ex(x_1,x_2).\]
An element $x\in \Lambda$ is called \emph{rigid} if $\ex(x,x)=0$ and $C\in \co$ is called rigid of it contains a rigid element. 
We denote the natural maps $p_\pm\colon \Lambda(\did)\ra E^\pm(\did)$ and recall the classical bijections \[\lambda_\pm\colon\mq\ra \co(\did),\: \lambda_\pm(\fm)\coloneq \overline{p_\pm^{-1}(\Z_\pm(\fm))}.\]
We will denote by $C\coloneq \lambda_+^{-1}$.
Recall that the involution $(-)^*\coloneq \lambda_-^{-1}\lambda_+\colon\mq\ra\mq$ is nothing but the Aubert-Zelevinsky involution, see \cite{MoeWal}, \cite{Zel}, \cite{Aub}.
\subsection{Coxeter functors}\label{S:coxeter}
We refer for the contents of this section to \cite[§9]{AizLap}, see also \cite[§2.3]{LapMin25}.
Recall for $\Pi$-modules $x,y$ the map
\[\mathcal{T}_{x,y}\colon  \mathrm{Hom}_{Q^+}(x,y)\ra\mathrm{Hom}_{Q^+}(x,\tau(y))\cong \mathrm{Ext}_{Q^+}^1(x,y)^*,\]
where $\tau$ is the Coxeter-functor of $Q^+$. To be more precise, $\tau$ sends a representation $(V,T)$ of $Q^+$ to $(V',T')$, where $V_i'=V_{i-1},T_i'=T_{i-1}$.
There exists an exact sequence, \emph{cf.} \cite[§8]{GeisLecSchr07}, \cite[Proposition 9.1]{AizLap},
\[\begin{tikzcd}
    0\arrow[r]&\Hp(x,y)\arrow[r]&\mathrm{Hom}_{Q^+}(x,y)\arrow[r,"\mathcal{Tau}_{x,y}"]&\mathrm{Ext}_{Q^+}^1(x,y)^*
    \arrow[r]&\Ex(x,y)\arrow[d]\\ &0&\arrow[l]\Hp(x,y)^*&\arrow[l]\mathrm{Ext}_{Q^+}^1(x,y)^*&\arrow[l,"\mathcal{Tau}_{x,y}^*"]\mathrm{Ext}_{Q^+}^1(x,y)^*
\end{tikzcd}\]
In particular,
\[\dim_\C\Hp(x,y)=\dim_\C\mathrm{ker}(\mathcal{Tau}_{x,y})\]
and
\[\dim_\C\Hp(x,y)-\alpha_+(x,y)=\dim_\C\mathrm{coker}(\mathcal{Tau}_{x,y}).\]
For $C,D\in\co$ we denote by $\dim_\C\mathrm{coker}(\mathcal{Tau}_{C,D})$ the minimum of the attained dimensions.
\subsection{Saturated and reduced components}
We call $\fa\in \mq$ \emph{basic} if it is of one of the following forms.
\[\fa=[a,b]\text{ or }\fa=[a,a]+\ldots+[b,b].\]
A multisegment $\fm$ is called $\fa$-saturated if it is of the form
\[\fm=[a_1,b]+\ldots+[a_k,b],\, \ain{a_i}{a}{b}\]
in the first case, and of the form
\[\fm=[a,a]+\ldots+[b_1,b_1]+\ldots+[a,a]+\ldots+[b_k,b_k],\, \ain{b_i}{a}{b}\] in the second case.
The component $C(\fa)$ is rigid and we pick a rigid $x_\fa\in C(\fa)$. A subcomponent of $C(\fa)$ is any $C'\in \co$ containing a submodule of $C(\fa)$.

An element $x\in \Lambda$ is called $\fa$-reduced if 
\[\Hp(x,x_\fa)=0\] for some $x_\fa\in C(\fa)$ and $\fa$-saturated if for all submodules $x'$ of $x$
\[\Hp(x',x_\fa)\neq 0\] for some $x_\fa\in C(\fa)$.
\begin{lemma}[{\cite[Lemma 7.5]{LapMin25}}]
    Any $\fa$-saturated module $x$ is a direct sum of irreducible submodules of $x_\fa$, which are all rigid.
\end{lemma}
An irreducible component $C\in \co$ is called $\fa$-saturated if it contains an $\fa$-saturated module and $\fa$-reduced if it does not contain one.
\begin{lemma}[{\cite[Lemma 7.6]{LapMin25}}]
    Every $C\in \co$ can be uniquely written as 
    \[C=C_1*C_2\] with $C_1$ $\fa$-saturated and $C_2$ $\fa$-reduced.
\end{lemma}
We refer to this as the $\fa$-decomposition of $C$ and denote in the above lemma 
$C_1=\fa(C)$ and $C_2=\Dde(C)$.
\subsection{Regular multisegments}\label{S:regular}
A segment $\De_1=[a,b]$ precedes a segment $\De_2=[c,d]$, denoted by
$\De_1\prec\De_2$, if
\[a+1\le c\le b+1\le d.\]
We call $\De_1$ and $\De_2$ unlinked if neither $\De_1\prec\De_2$ nor $\De_2\prec\De_1$.
We set
\[l([a,b])\coloneq b-a+1,\,\sh{[a,b]}=[a-1,b-1],\, [a,b]^\lor=[-b,-a]\] and extend those notions linearly to multisegments.
We also write \[a([a,b])=a,\, b([a,b])=b.\]

We call a multisegment $\fm=\De_1+\ldots+\De_k\in\mq$ \emph{regular}, if for all $\ain{i\neq j}{1}{k}$, $b(\De_i)\neq b(\De_j),\, a(\De_i)\neq a(\De_j)$.
We recall that a regular multisegment is said to be of type $4231$ if $k\ge 4$ and
\begin{enumerate}
    \item $\De_{i}\prec\De_{i-1},\,\ain{i}{3}{k}$,
    \item $a(\De_k)<a(\De_1)<a(\De_{k-1})$,
    \item $b(\De_3)<b(\De_1)<b(\De_{2})$.
\end{enumerate}
Similarly, it is said to be of type $3412$
if $k\ge 4$ and
\begin{enumerate}
    \item $\De_{i}\prec\De_{i-1},\,\ain{i}{4}{k}$ and $\De_2\prec\De_1$,
    \item $a(\De_2)<a(\De_k)<a(\De_1)<a(\De_{k-1})$,
    \item $b(\De_4)<b(\De_1)<b(\De_3)<b(\De_{2})$.
\end{enumerate}
A regular multisegment is called \emph{balanced} if it does not contain any submultisegment of type $4231$ or $3412$.
\begin{theorem}[{\cite[Theorem 7.1]{MinLa18}}]
    Let $\fm$ be a regular multisegment. Then $C(\fm)$ is rigid if and only if $\fm$ is balanced.
\end{theorem}
\begin{lemma}[{\cite[Lemma 8.2]{LapMin25}}]\label{L:bala}
    Let $\fm\in \mq$ be a balanced multisegment of length at least $2$. Then there exists non-zero basic $\fa\in \mq$ such that the following holds.
    \begin{enumerate}
        \item $\fa(C(\fm)),\, \Dde(C(\fm))\neq 0$.
        \item $\Dde(C(\fm))=C(\fm')$ and $\fm'$ is balanced.
        \item $C(\fm)$ strongly commutes with every subcomponent of $C(\fa)$.
    \end{enumerate}
\end{lemma}
\begin{lemma}\label{L:indcomp}
    Let $C,D\in \co$ and $\fa$ a basic multisegment such that for each subcomponent $C'$ of $C(\fa)$ the components $C'$ and $C$ strongly commute. Then
    \[\hp(D,C)=\hp(\Dde(D),\Dde(C))+\hp(\fa(D),C).\]
\end{lemma}
\begin{proof}
    Let \[0\ra x_2\ra x\ra x_1\ra 0\] be a short exact sequence with suitably generic $x_1\in \fa(D),x_2\in \Dde(D)$ and $x\in D$.
    For a generic $y\in C$ we obtain an exact sequence
    \[0\ra \Hp(x_1,y)\ra \Hp(x,y)\ra \Hp(x_2,y)\ra \Ex(x_1,y).\]
    By assumption the last term vanishes for suitable choices of $x$ and hence
    \[\hp(D,C)=\hp(\fa(D),C)+\hp(\Dde(D),C).\]
    Since $\Dde(D)$ is $\fa$-reduced, and $\fa$-reduced and $\fa$-saturated modules form a torsion pair, see \cite[§7]{LapMin25}, $\hp(\Dde(D),C)=\hp(\Dde(D),\Dde(C))$.
\end{proof}
\section{Representation theory}
Let $\DD$ be a central division algebra of degree $d$ over $\Ff$.
We set $G_n\coloneq \gl_n(\DD)$ and denote by $\rep_n$ the category of complex, smooth, finite length representations of $G_n$. If $\pi\in \rep_n$ write $\deg(\pi)=n$. For any partition $\alpha$ of $n$, we denote the $\DD$-points of the parabolic subgroup of $G_n$ containing the upper triangular matrices of type $\alpha$ by $P_\alpha$ and its unipotent part by $U_\alpha$. The opposite parabolic of $P_\alpha$ is denoted by $\overline{P_\alpha}$ and it is conjugate to $P_{\overline{\alpha}}$, where $\overline{\alpha}$ is the opposite parititon of $\alpha$. We also let
\[w_\alpha\coloneq\begin{pmatrix}
    0&\ldots&1_{\alpha_1}\\\vdots&\iddots&\vdots\\1_{\alpha_k}&\ldots&0
\end{pmatrix}.\]
We let $\Irr_n$ be the set of isomorphism classes of irreducible representations in $\rep_n$, $\rep=\bigcup_{n\in\NN}\rep_n$ and $ \irr=\bigcup_{n\in\NN}\irr_n$.
We denote the socle and cosocle of $\pi\in \rep$ by $\soc(\pi)$ and $\cosoc(\pi)$ and say $\pi$ is SI if $\soc(\pi)$ is irreducible and $\soc(\pi)$ appears with multiplicity one in the decomposition series of $\pi$.

For $\alpha=(\alpha_1,\ldots,\alpha_k)$ a partition of $n$, we denote the normalized parabolic induction functor $\id_{P_\alpha}$ from $P_\alpha$ to $G_n$ by
\[(-)\times\ldots\times(-)\colon \rep_{\alpha_1}\times\ldots\times\rep_{\alpha_k}\ra\rep_n,\] the category of complex, smooth, finite length representations of $G_{\alpha_1}\times\ldots\times G_{\alpha_k}$ by $\rep_\alpha$ and the normalized Jacquet functor
\[r_\alpha\colon \rep_n\ra \rep_\alpha.\]

Furthermore, we denote by $\KK_n$ the Grothendieck group of $\rep_n$ and write the natural map as $[-]\colon\rep_n\ra \KK_n$.
We recall that a cuspidal representation $\rho\in \Irr_n$ is a representation such that for all non-trivial partitions $\alpha$ of $r_\alpha(\rho)=0$. We denote the corresponding subset as $\cus_n$ and $\cus\coloneq\bigcup_{n\in\NN}\cus_n$.

We recall, see for example \cite[§4.5]{MinguezSecherreStevens2014} that for $\rho\in \cus$, there exists a unique unramified character $\nur=\lvert-\rvert^{s_\rho}, s_\rho\in\mathbb{Z}_{>0}$ of $G_n$ such that the representation $\rho\times\rho\chi$, $\chi$ another character, is reducible if and only if $\chi\cong\nur^\pm$. We denote the cuspidal line of $\rho$ by $\ZZ[\rho]\coloneq\{[\rho\nur^k]:\,k\in \ZZ\}$.

For a set $S$, we let $\NN(S)$ be the monoid of formal finite sums of elements in $S$. We recall the cuspidal support map,
\[\cusp:\Irr\ra \NN(\cus),\] given by $\cusp(\pi)= [\rho_1]+\ldots+[\rho_k]$, if $\pi$ is a subquotient of $\rho_1\times\ldots\times\rho_k$ and $\rho_i\in\cus$ for all $\ain{i}{1}{k}$.

We note that for all $\pi,\pi'\in \rep$, $[\pi\times\pi']=[\pi\times \pi']$. If $\pi\in \Irr$, twisting the action by conjugation of a fixed element in the linear group preserves the isomorphism class and twisting the action by $g\mapsto {}^tg^{-1}$ sends $\pi$ to $\pi^\lor$. As a corollary we obtain the following lemma.
\begin{lemma}\label{L:easylemma}
    Let $\pi_1,\ldots,\pi_k\in \Irr$. If \[\dim_\C\ho_{G_{\deg(\pi_1\times\ldots\times\pi_k)}}(\pi_k\times\ldots\times\pi_1,\pi_1\times\ldots\times\pi_k,)=1,\] then $\soc(\pi_1\times\ldots\times\pi_k)$ is irreducible.
\end{lemma}
\subsection{Multisegments}
Let $a,b\in \ZZ$, $a\le b$, $\rho\in \cus$. To this datum we associate the segment $[a,b]_\rho\coloneq ([\rho \nur^a],\ldots,[\rho\nur^b])$.
We let \[l[a,b]_\rho)=b-a+1,\, \deg([a,b]_\rho=(b-a+1)\deg(\rho),\, [a,b]_\rho^\lor=[-b,-a]_{\rho^\lor}.\]
The set of segments is denoted by $\Se$ and the set of multisegments
is defined as $\Ms\coloneq \NN(\Se)$. We extend the functions $l$, $\deg$ and $(-)^\lor$ linearly to $\Ms$.

For fixed $\rho\in\cus$, we let $\Ms(\rho)$ be the set of multisegments of the form $[a_1,b_1]_\rho+\ldots+[a_k,b_k]_\rho$. For $\fm$ we write the block decomposition of $\fm=\fm_1+\ldots+\fm_k$ with $\fm_i\in \Ms(\rho_i)$ and $[\rho_i]\notin \ZZ[\rho_j]$ for $i\neq j$.
It is clear that $\Ms(\rho)$ is in natural bijection with $\mq$. We thus can import notions like \emph{precede} or the Aubert-Zelevinsky involution to $\Ms$. In particular, we call $\fm\in \Ms$ balanced if each $\fm_i$ in its block decomposition is balanced, and similarly for regular.

For a multisegment $\De_1+\ldots+\De_k$, we say $(\De_1,\ldots,\De_k)$ is an arranged form if there exists no $i<j$ such that $\De_i\prec\De_j$.
\begin{theorem}[\cite{Zel}]\label{T:zel}
    There exists a bijection
    \[\Z\colon \Ms\ra\Irr\] satisfying the following properties.
    \begin{enumerate}
        \item $\Z([a,b]_\rho)=\soc(\rho\nur^a\times\ldots\times\rho\nur^b)$.
        \item For $m,n\in\NN,\,m+n=\deg([a,b]_\rho)$,
        \[r_{m,n}(\Z([a,b]_\rho))=0\] unless $\deg(\rho)\lvert m$, in which case, \[r_{m,n}(\Z([a,b]_\rho))=\Z([a,a+\frac{m}{\deg(\rho)}-1]_\rho)\otimes \Z([a+\frac{m}{\deg(\rho)}-1,b]_\rho).\]
        \item If $(\De_1,\ldots,\De_k)$ is an arranged form of $\fm\in\Ms$, then
        \[\Z(\fm)=\soc(\Z(\De_1)\times\ldots\times \Z(\De_k))=\cosoc(\Z(\De_k)\times\ldots\times\Z(\De_1))\]
        and appears with multiplicity $1$ in $\Z(\De_1)\times\ldots\times \Z(\De_k)$.
        \item $\deg(\Z(\fm))=\deg(\fm)$.
        \item If $\fm=\De_1+\ldots+\De_k$ is a multisegment such that the $\De_i$ are pairwise unlinked, then 
        \[\Z(\fm)=\Z(\De_1)\times\ldots\times \Z(\De_k).\]
        \item If $\fm'$ is a second multisegment, $\Z(\fm+\fm')$ appears in $\Z(\fm)\times \Z(\fm')$ with multiplicity $1$.
        \item If $\fm\in \Ms(\rho),\,\fm'\in \Ms(\rho')$ with $\rho'\notin\ZZ[\rho]$, \[\Z(\fm+\fm')\cong \Z(\fm)\times \Z(\fm').\]
    \end{enumerate}
\end{theorem}
As is common, we will usually write $\L(\fm)\coloneq\Z(\fm^*)$.
\subsection{Basic representations}
We call $\fa\in \Ms$ \emph{basic} if it is of one of the following forms.
\[\fa=[a,b]_\rho\text{ or }\fa=[a,a]_\rho+\ldots+[b,b]_\rho.\]
A multisegment $\fm$ is called $\fa$-saturated if it is of the form
\[\fm=[a_1,b]_\rho+\ldots+[a_k,b]_\rho,\, \ain{a_i}{a}{b}\]
in the first case and of the form
\[\fm=[a,a]_\rho+\ldots+[b_1,b_1]_\rho+\ldots+[a,a]_\rho+\ldots+[b_k,b_k]_\rho,\, \ain{b_i}{a}{b}\] in the second case.
A component of $\fa$ is a multisegment of the form $[a',b],\,\ain{a'}{a}{b}$ in the first case and 
$[a',a']+\ldots+[b,b],\,\ain{a'}{a}{b}$ in the second case. In particular, if $\fm$ is an $\fa$-saturated segment,
\[\Z(\fm)=\Z(\fa_1)\times\ldots\times\Z(\fa_k),\]
where the $\fa_i$ are basic multisegments of the same type as $\fa$.
A representation $\pi\in \irr$ is called $\fa$-reduced if there exists no $\tau\in \irr$ and an $\fa$-saturated multisegment $\fm$ such that $\pi\hra \Z(\fm)\times \tau$. In \cite[§7]{LapMin25} it is shown that this is equivalent to the fact that there exists
no $\tau\in \irr$ and an $\fa$-saturated multisegment $\fm$ such that $r_{\deg(\fm),\deg(\tau)}(\pi)$ contains $\Z(\fm)\otimes \tau$ as a subquotient.
\begin{lemma}[{\cite[§7]{LapMin25}}]\label{L:indint}
Let $\fa$ be a basic representation.
    For every irreducible representation $\pi=\Z(\fn)$, there exists a unique, possible $0$, $\fa$-saturated multisegment $\fm$ and a unique $\fa$-reduced representation $\Z(\fn')$ such that $\pi=\soc(\Z(\fm)\times\Z(\fn'))$ and $\pi$ appears with multiplicity one in the composition series of $\Z(\fm)\times \Z(\fn')$.

    If $\fn\in \Ms(\rho)\cong \mq$, then
\[\fa(C(\fn))=C(\fm),\, \Dde(C(\fn))=C(\fn').\]
\end{lemma}
We write in this case $\fm=\fa(\fn)$, $\fn'=\Dde(\fm)$ and call it the $\fa$-decomposition of $\fn$.
\begin{lemma}[{\cite[§7]{LapMin25}}]
    Let $\fn\in \Ms(\rho)$ be a balanced multisegment of length at least $2$. Then there exists a basic multisegment $\fa$ such that the following holds.
    \begin{enumerate}
        \item The composition is non-trivial, \emph{i.e.} $\fa(\fn),\,\Dde(\fn)\neq 0$.
        \item For all components $\fa'$ of $\fa$, $\Z(\fa')\times\Z(\fn)$ is irreducible.
        \item The decomposition preserves being balanced, \emph{i.e.} $\Dde(\fn)$ is balanced.
    \end{enumerate}
\end{lemma}
\section{Intertwining operators}\label{S:intop}
Let $\pi,\sigma\in\rep$ and $s\in \C$. We recall the intertwining operators \[M_{\pi,\sigma}(s)\colon \sigma\abs^s\times\pi\abs^{-s}\ra\pi\abs^{-s}\times\sigma\abs^s\]
defined by 
\[f\mapsto\left(g\mapsto \int_{U_{\deg(\pi),\deg(\sigma)}}f(w_{\deg(\sigma),\deg(\pi)}ug)\,\mathrm{d}u\right),\] which converges for $\mathrm{Re}(s)>>0$. One can continue $M_{\pi,\sigma}(s)$ meromorphically to the whole complex plane and we denote by $\LL(\sigma,\pi)$ the order of the pole of the intertwining operator
$M_{\pi,\sigma}(s)$ at $s=0$.
It follows that \[M_{\pi,\sigma}\coloneq \restr{s^{\Lambda(\pi,\sigma)}M_{\pi,\sigma}(s)}{s=0}\colon \sigma\times\pi\ra\pi\times\sigma\]
defines a non-zero morphism.
Similarly, we let $\alpha(\pi,\sigma)$ the order of zero of
\[M_{\sigma,\pi}(s)\circ M_{\pi,\sigma}(s)\] at $s=0$.
We recall the following properties.
\begin{theorem}[{\cite[§IV]{Wal03}, \cite[Lemma 2.3]{MinLa18}, \cite{KKKO}}]\label{T:intertwin}
Let $\pi,\pi',\sigma,\sigma'\in\rep$.
\begin{enumerate}
    \item If $\sigma$ is irreducible, then
    \[\Lambda(\pi\times\pi',\sigma)=\LL(\pi,\sigma)+\LL(\pi',\sigma).\]
    \item If $\sigma'\hra \sigma, \pi'\hra \pi$, there exists $\lambda\in\C$ such that the following diagram commutes.
    \[\begin{tikzcd}
        \sigma\times \pi\arrow[rr,"M_{\pi,\sigma}"]&&\pi\times\sigma\\
        \sigma'\times\pi'\arrow[rr,"\lambda M_{\pi',\sigma'}"]\arrow[u,hookrightarrow]&&\pi'\times\sigma'\arrow[u,hookrightarrow]
    \end{tikzcd}\]
    Moreover, $\lambda\neq 0$ if and only if $\Lambda(\pi,\sigma)=\Lambda(\pi',\sigma')$.
    \item We have \[\alpha(\pi,\sigma)=\alpha(\sigma,\pi)\] and for any subquotient $\omega$ of $\pi\times\pi'$ we have that 
    \[\alpha(\omega,\sigma)=\alpha(\pi,\sigma)+\alpha(\pi',\sigma).\]
    \item If $\pi$ and $\sigma$ are subquotients of representations induced from irreducible representations, the map $M_{\pi,\sigma}\circ M_{\sigma,\pi}$ is a scalar.
    \item We have $\Lambda(\pi,\sigma)=\Lambda(\sigma^\lor,\pi^\lor)$ and $\alpha(\pi,\sigma)=\alpha(\pi^\lor,\sigma^\lor)$.
\end{enumerate}
\end{theorem}

For $\pi_1,\pi_2\in \Irr$, we define
\[\mathfrak{d}(\pi_1,\pi_2)\coloneq \Lambda(\pi_1,\pi_2)+\Lambda(\pi_2,\pi_1)+\alpha(\pi_1,\pi_2)\ge 0.\]
We say $\pi_1$ and $\pi_2$ \emph{strongly commute} if $\fd(\pi_1,\pi_2)=0$. The following is a direct consequence of \Cref{T:intertwin}.
\begin{lemma}\label{L:strcom}
    Let $\pi_1,\pi_2\in\irr$. The following are then equivalent.
    \begin{enumerate}
        \item The representations $\pi_1,\pi_2$ strongly commute.
        \item $M_{\pi_1,\pi_2}$ is an isomorphism.
        \item $M_{\pi_2,\pi_1}$ is an isomorphism.
        \item $M_{\pi_1,\pi_2}\circ M_{\pi_2,\pi_1}$ is non-zero.
    \end{enumerate}
\end{lemma}
\begin{lemma}\label{L:basecase}
Let $\rho,\rho'\in \cus$. Then 
\[ \alpha(\rho,\rho')=\begin{cases}
    -2&\text{if }\rho'\cong \rho,\\1&\text{if }\rho\cong \rho\nur^{\pm1},\\0&\text{otherwise}.
\end{cases}\]\[\LL(\rho,\rho')=\begin{cases}
    1&\text{if }\rho'\cong \rho,\\0&\text{otherwise}.
\end{cases}.\]
\end{lemma}
\begin{proof}
    For the computation of $\Lambda$ see for example \cite[Proposition 7.5]{Dat}.
    By \Cref{L:strcom} and the fact that $\rho\times\rho'$ is reducible only if $\rho'\cong\rho\nur^{\pm1}$ implies that $\alpha(\rho,\rho)=-2$ and $\alpha(\rho,\rho')=0$ for all other $\rho'$ different from $\rho\nur^{\pm1}$. Moreover, the Geometric Lemma of Bernstein-Zelevinsky implies that there is only , up to a scalar, one morphism $\rho\times\rho\nur\ra\rho\nur\times\rho$, which has image the unique irreducible subrepresentation of $\rho\nur\times\rho$. Hence $\alpha(\rho,\rho\nur)>0$. By Silberger's formula for $\alpha$, see for example \cite[§8.3, §8.4]{Dat}, $\alpha(\rho,\rho\nur)\le 1$ and hence the claim follows.
\end{proof}
\begin{theorem}[{\cite[§2]{MinLa18}}]\label{T:si}
A representation $\pi\in \Irr$ is called $\square$-irreducible if one of the following equivalent conditions holds.
\begin{enumerate}
    \item $\pi\times\pi\in \Irr$.
    \item $M_{\pi,\pi}$ is a scalar.
    \item For all $\sigma\in \Irr$, $\pi\times\sigma$ is SI.
    \item For all $\sigma\in \Irr$, $\sigma\times \pi$ is SI.
\end{enumerate}
\end{theorem}
Let us note the following easy lemma.
\begin{lemma}
    Let $\pi\in \Irr$. Then $\pi$ is $\square$-irreducible if and only if
    for all $\sigma\in \Irr$, \[\dim_\C\ho_{G_{\deg(\pi\times\sigma)}}(\pi\times\sigma,\sigma\times\pi)=\dim_\C\ho_{G_{\deg(\pi\times\sigma)}}(\sigma\times\pi,\pi\times\sigma)=1.\]
\end{lemma}
\begin{proof}
    Plugging in $\sigma=\pi$, yields that the condition on the Hom-spaces implies that $M_{\pi,\pi}$ is a scalar and thus $\pi$ is $
    \square$-irreducible. On the other hand, if $\pi\in \irs$, and one of the Hom-spaces would be more than one-dimensional, $\soc(\pi\times\sigma)$ or $\soc(\sigma\times\pi)$ would appear with multiplicity greater than $1$ on the composition series of $\pi\times\sigma$.
\end{proof}

\begin{lemma}[{\cite[Theorem 7.1]{MinLa18}}]
    Let $\fm\in \Ms$ be regular. Then $\Z(\fm)$ is $\square$-irreducible if and only if it is balanced.
\end{lemma}
We denote the subset of $\Irr$ consisting of $\square$-irreducible representations by $\irs$. 
\begin{lemma}\label{L:fd0}
    Let $\pi\in\irs$ and $\sigma\in \Irr$. Then $\pi\times\sigma$ is irreducible if and only if $\fd(\pi,\sigma)=0$.
\end{lemma}
\begin{proof}
    Note that $\fd(\pi,\sigma)=0$ if and only if $M_{\sigma,\pi}\circ M_{\sigma,\pi}$ is a scalar by \Cref{T:intertwin}. From \Cref{T:si} it follows that $\cos(\pi\times\sigma)=\soc(\pi\times \sigma)$ and hence $\pi\times\sigma\in \irr$.
\end{proof}
\begin{lemma}[{\cite[Lemma 2.10]{MinLa18}}]\label{L:prodsquare}
    Let $\pi,\pi'\in\irs$ such that $\pi\times\pi'\in\Irr$. Then $\pi\times\pi'\in\irs$.
\end{lemma}
Let $\pi_1,\pi_2\in \rep$. We let $(\pi_2\times\pi_1)_{P_{\deg(\pi_1),\deg(\pi_2)}}$ be the $P_{\deg(\pi_1),\deg(\pi_2)}$-subrepresentation of $\pi_2\times\pi_1$ consisting of functions $f\in \pi_2\times\pi_1$ whose (compact) support is contained in \[P_{\deg(\pi_2),\deg(\pi_1)}w_{\deg(\pi_2),\deg(\pi_1)}P_{\deg(\pi_1),\deg(\pi_2)}.\] By the Geometric Lemma of Bernstein and Zelevinsky, \emph{cf.} \cite{Ber}, we have that
\[r_{\deg(\pi_1),\deg(\pi_2)}((\pi_2\times\pi_1)_{P_{\deg(\pi_1),\deg(\pi_2)}})\cong \pi_1\otimes\pi_2\]
and this isomorphism depends only on the choice of an Haar-measure on $U_{\deg(\pi_1),\deg(\pi_2)}$.
We denote the so obtained inclusion \[\iota_{\pi_1,\pi_2}\colon \pi_1\otimes\pi_2\hra r_{\deg(\pi_1),\deg(\pi_2)}(\pi_2\times\pi_1).\]
\begin{lemma}\label{L:polevanish}
    Let $\pi_1,\pi_2\in \Irr$ such that $\pi_1\otimes \pi_2$ appears in \[r_{\deg(\pi_1),\deg(\pi_2)}(\pi_2\times \pi_1)\] with multiplicity $1$, then $\Lambda(\pi_1,\pi_2)=0$.
\end{lemma}
\begin{proof} The argument is essentially the one given in \cite[p. 285]{Wal03} and we repeat it for the sake of the reader.
Let us note that the restriction of $M_{\pi_1,\pi_2}$ does not vanish on $(\pi_2\times\pi_1)_{P_{\deg(\pi_1),\deg(\pi_2)}}$ only if $\Lambda(\pi_1,\pi_2)=0$, see \cite[p. 283]{Wal03}. Indeed, by construction, the restriction of $M_{\pi_1,\pi_2}(s)$ to $(\pi_2\abs^s\times\pi_1\abs^{-s})_{P_{\deg(\pi_1),\deg(\pi_2)}}$ does not vanish for generic $s$ and the intertwining operator is uniquely determined (up to the choice of Haar-measure) by 
\[M_{\pi_1,\pi_2}(s)(f)(1)=\int_{U_{\deg(\pi_1),\deg(\pi_2)}}f(w_{\deg(\pi_2),\deg(\pi_1)}u)\,\mathrm{d}u,\, f\in (\pi_2\abs^s\times\pi_1\abs^{-s})_{P_{\deg(\pi_1),\deg(\pi_2)}}. \]
In particular, if $M_{\pi_1,\pi_2}(s)$ would have a pole at $s=0$, the renormalisation $M_{\pi_1,\pi_2}$ would vanish on $(\pi_2\times\pi_1)_{P_{\deg(\pi_1),\deg(\pi_2)}}$.

By Frobenius reciprocity, we have a map \[M_{\pi_1,\pi_2}(-)(1)\colon r_{\deg(\pi_1),\deg(\pi_2)}(\pi_2\times\pi_1)\ra \pi_1\otimes \pi_2\] and the claim that $M_{\pi_1,\pi_2}$ does not vanish on $(\pi_2\times\pi_1)_{P_{\deg(\pi_1),\deg(\pi_2)}}$ is equivalent to the claim that the composition 
\[\pi_1\otimes\pi_2\cong r_{\deg(\pi_1),\deg(\pi_2)}((\pi_2\times\pi_1)_{P_{\deg(\pi_1),\deg(\pi_2)}})\hra r_{\deg(\pi_1),\deg(\pi_2)}(\pi_2\times\pi_1)\ra \pi_1\otimes \pi_2\] does not vanish.
Since by assumption $\pi_1\otimes\pi_2$ appears with multiplicity $1$ in \[r_{\deg(\pi_1),\deg(\pi_2)}(\pi_2\times\pi_1),\] the above composition does not vanish and the claim follows.
\end{proof}
\begin{corollary}\label{C:block}
    Let $\fm,\fn\in \Ms$ with block decomposition $\fm=\fm_1+\ldots+\fm_k$ and $\fn=\fn_1+\ldots+\fn_k$, where we allow some of the $\fm_i,\fn_i$ to be $0$ and $\fm_i$ and $\fn_i$ are in the same block.
Then
    \[\Lambda(\Z(\fm),\Z(\fn))=\prod_{i=1}^k\Lambda(\Z(\fm_i),\Z(\fn_i)),\]
        \[\alpha(\Z(\fm),\Z(\fn))=\prod_{i=1}^k\alpha(\Z(\fm_i),\Z(\fn_i)).\]
\end{corollary}
It is easy to see from \Cref{T:intertwin} and \Cref{L:basecase} that for $\fm,\fn\in \Ms(\rho)$
\[\alpha(\Z(\fm),\Z(\fn))=-(\gdim(\fm),\gdim(\fn)).\]
\subsection{Aubert-Zelevinsky involution}\label{S:AZ}
In this section we recall the Aubert-Zelevinksy involution and check how it behaves with respect to intertwining operators. In particular, we will see that it mirrors what we expect on the geometric side.
Let $\rho_1,\ldots,\rho_n\in \cus_t$ and let $\rep_\rho,\,\rho=[\rho_1]+\ldots+[\rho_k]$ be the full subcategory of $\rep$ consisting of representations such that each irreducible subquotient has cuspidal support $[\rho_1]+\ldots+[\rho_k]$. Then there exists an exact involution, the Aubert-Zelevinsky involution, \emph{cf.} \cite{Zel}, \cite{Aub}, \cite{MoeWal},
    \[(-)^*\colon \rep_\rho\ra\rep_\rho\] such that $\Z(\fm)^*\cong\Z(\fm^*)$.
We also recall the cohomological dual, \emph{c.f.} \cite{BerBerKaz}, \cite{SchneiStu},
\[D\colon \rep_\rho\ra\rep_\rho,\]
which satisfies the following. Let $\alpha$ be a partition.
\begin{enumerate}
\item If $\pi$ is irreducible, there exists a canonical isomorphism $D(\pi)\cong(\pi^*)^\lor$.
    \item $D\circ r_\alpha=r_\alpha \circ D$.
    \item $D\circ \id_{\alpha}=\id_{{\alpha}}\circ D$.
    \item $D$ commutes with Frobenius reciprocity and with $\iota_{\pi_1,\pi_2}\colon\pi_1\otimes\pi_2\hra r_{\deg(\pi_1),\deg(\pi_2)}(\pi_2\times\pi_1)$ for any $\pi_1,\pi_2\in\rep$.
\end{enumerate}
For the last item, we refer to \cite[§6.3]{BezKaz15}.
\begin{lemma}\label{L:aubinv}
     Let $\fm,\fn\in \Ms$. Then
    \[\alpha(\Z(\fm),\Z(\fn))=\alpha(\L(\fn),\, \L(\fm)),\,\Lambda(\Z(\fm),\Z(\fn))=\Lambda(\L(\fn),\L(\fm)).\]
\end{lemma}
\begin{proof}
    Since $\alpha$ only depends on the cuspidal support and the involution leaves that unchanged, the claim for $\alpha$ follows immediately.

    For the claim regarding the pole, we first note that it is easy to see that $\Lambda(\Z(\fm),\Z(\fn))=\Lambda(\Z(\fn)^\lor,\Z(\fm)^\lor)$. It therefore suffices to show that \[D( M_{\Z(\fm),\Z(\fn)}(s))=M_{D(\Z(\fm)),\, D(\Z(\fn))}(s)\] for generic $s$. To check this we follow the arguments of \cite[§2]{Dat}.
    Namely, we let \[M'\colon r_{\deg(\fm),\deg(\fn)}(\Z(\fn)\abs^s\times\Z(\fm)\abs^{-s})\ra\Z(\fm)\abs^{-s}\otimes\Z(\fn)\abs^s\] be the map obtained from $M_{\Z(\fm),\Z(\fn)}(s)$ by Frobenius reciprocity. As in the proof of \Cref{L:polevanish}, we see that up to a choice of the Haar-measure the map $M'$ is uniquely determined by demanding that the composition of $M'$ with the inclusion \[\iota_{\Z(\fm)\abs^{-s},\Z(\fn)\abs^s}\colon\Z(\fm)\abs^{-s}\otimes\Z(\fn)\abs^s\hra r_{\deg(\fm),\deg(\fn)}(\Z(\fn)\abs^s\times\Z(\fm)\abs^{-s})\] is the identity. Since $D$ commutes with Frobenius reciprocity, we have that $D\circ M'$ is the map corresponding by Frobenius reciprocity to $D(M_{\Z(\fm),\Z(\fn)}(s))$. Using the remaining commuting properties of $D$, we see on the other hand that $D\circ M'$ is the map obtained by Frobenius reciprocity also of $=M_{D(\Z(\fm)),\, D(\Z(\fn))}(s)$, thus proving the equality.
\end{proof}
\subsection{Computation of poles}\label{S:CoP}
We fix a cuspidal representation $\rho\in\cus$ for the rest of the section.
The following theorem is inspired by \cite[Corollary 7.4]{AizLap}.
\begin{lemma}\label{L:central}
Let $\pi_1\in \irr$ and $\fa\in \Ms(\rho)$ a basic segment such that:
\begin{enumerate}
    \item For all components $\fa'$ of $\fa$, the representation $\pi_1\times \Z(\fa')$ is irreducible.
    \item The representations $\pi_1$ and $\Dde(\pi_1)$ are $\square$-irreducible.
\end{enumerate}
 Then we have for all $\pi_2\in\irr$ \[\Lambda(\pi_1,\pi_2)=\Lambda(\Dde(\pi_1),\Dde(\pi_2))+\Lambda(\pi_1,\fa(\pi_2)).\]
\end{lemma}
\begin{proof}
    By \Cref{T:intertwin}, there exists $\lambda\in \C$ such that the following diagram commutes
\[\begin{tikzcd}
    \fa(\pi_2)\times \Dde(\pi_2)\times\pi_1\arrow[rr,"M_{\pi_1,\fa(\pi_2)\times\Dde(\pi_2)}"]&&\pi_1\times \fa(\pi_2)\times \Dde(\pi_2)\\
    \pi_2\times\pi_1\arrow[rr,"\lambda M_{\pi_1,\pi_2}"]\arrow[u,hookrightarrow]&&\pi_1\times \pi_2\arrow[u,hookrightarrow]
\end{tikzcd}\]
    and $\lambda\neq 0$ if and only if \[\Lambda(\pi_1,\pi_2)=\Lambda(\pi_1,\Dde(\pi_2))+\Lambda(\pi_1,\fa(\pi_2)).\]
    By \Cref{L:prodsquare}, $\pi_1\times \fa(\pi_2)\in \irs$, thus the socle $\sigma\coloneq\soc(\pi_1\times \fa(\pi_2)\times \Dde(\pi_2))$ is irreducible. Hence $\sigma=\soc(\pi_1\times\pi_2)$ and if $\lambda=0$, it would appear as a subrepresentation with multiplicity $2$ in $\pi_1\times \fa(\pi_2)\times \Dde(\pi_2)$, a contradiction to \Cref{T:si}.
    It remains to show that $\Lambda(\pi_1,\Dde(\pi_2))=\Lambda(\Dde(\pi_1),\Dde(\pi_2))$. To see this, we will use the commutative diagram
    \[\begin{tikzcd}
    \Dde(\pi_2)\times\fa(\pi_1)\times\Dde(\pi_1)\arrow[rrr,"M_{\fa(\pi_1)\times\Dde(\pi_1),\Dde(\pi_2)}"]&&&\fa(\pi_1)\times\Dde(\pi_1)\times\Dde(\pi_2)\\
    \Dde(\pi_2)\times\pi_1\arrow[rrr,"\lambda' M_{\fa(\pi_1,\Dde(\pi_2)}"]\arrow[u,hookrightarrow]&&&\pi_1\times \Dde(\pi_2)\arrow[u,hookrightarrow]
\end{tikzcd}\] for a suitable $\lambda'\in \C$. We start by showing that $\lambda'\neq 0$. For this, it is, as in the first step, enough to show that the socle $\soc(\fa(\pi_1)\times\Dde(\pi_1)\times\pi_2)$ is irreducible, which by \Cref{L:easylemma} can be shown by proving that there exists a, up to a scalar, unique morphism 
\[\Dde(\pi_2)\times\Dde(\pi_1)\times \fa(\pi_1)\ra \fa(\pi_1)\times \Dde(\pi_1)\times\Dde(\pi_2).\]
By Frobenius reciprocity, the Geometric Lemma, and the fact that both $\Dde(\pi_1)$ and $\Dde(\pi_2)$ are $\fa$-reduced, we obtain that any such morphism induces a morphism
\[\Dde(\pi_2)\times\Dde(\pi_1)\ra \Dde(\pi_1)\times\Dde(\pi_2).\] The claim then follows from \Cref{T:si}. We therefore showed that 
 \[\Lambda(\pi_1,\Dde(\pi_2))=\Lambda(\Dde(\pi_1),\Dde(\pi_2))+\LL(\fa(\pi_1),\Dde(\pi_2)).\]
 By \Cref{L:polevanish}, $\LL(\fa(\pi_1),\Dde(\pi_2))$ vanishes and the claim follows.
\end{proof}
\begin{theorem}\label{T:main}
    Let $\fm,\fn\in \Ms(\rho)$ such that at least one of $\fm,\fn,\fm^*,\fn^*$ is balanced. Then
    \[\Lambda(\Z(\fm),\Z(\fn))=\hp(C(\fn),C(\fm)),\, \Lambda(\L(\fm),\L(\fn))=\hp(C(\fm),C(\fn)).\]
\end{theorem}
\begin{proof}
    The proof follows by induction, the base case being clear. Moreover, the second equality follows from the first via \Cref{L:aubinv}. For the induction step, we can apply \Cref{L:bala} and \Cref{L:indcomp} on the homological side and \Cref{L:indint} and \Cref{L:central} on the representation theoretical side.
\end{proof}
We come now to the main conjecture of the paper.
\begin{conjecture}\label{C:int}
    Let $\fm,\fn\in \Ms(\rho)$. Then
 \[ \Lambda(\Z(\fm),\Z(\fn))=\hp(C(\fn),C(\fm)).\]
\end{conjecture}
We say $C_{\fm,\fn}$ holds if the above equation is satisfied.

The following is an easy consequence of the
the Crawley-Boevey identity\[\ex(C,D)=\hp(C,D)+\hp(D,C)-(\gdim(C),\gdim(D)),\]
see also \Cref{L:strcom}.
\begin{prop}\label{C:conjint}
Let $\fm,\fn\in \Ms(\rho)$ and assume that $C_{\fm,\fn}$ and $C_{\fm,\fn}$ hold. Then
\[\ex(C(\fm),C(\fn))=\fdp(\Z(\fm),\Z(\fn)).\] In particular, the following are then equivalent.
\begin{enumerate}
    \item $C(\fm)$ and $C(\fn)$ strongly commute.
\item $\Z(\fm)$ and $\Z(\fn)$ strongly commute.
\end{enumerate}
\end{prop}
\begin{rem}

    (1) If \Cref{C:rigid} is true, then \Cref{C:int} is true for any $\square$-irreducible representation $\Z(\fm)$ and $\fn=\fm$. 
    
    (2) Note that $C(\fn)$ commutes strongly with $C(\fm)$ if and only if $C(\fn)*C(\fm)=C(\fn+\fm)$, see \cite[Corollary 9.4]{AizLap}.
    On the other hand, if $Z(\fm)$ is $\square$-irreducible, then $\Z(\fm)\times \Z(\fn)=\Z(\fm+\fn)$ if and only if $\fdp(\Z(\fm),\Z(\fn))=0$ by \Cref{L:fd0}. Thus, the equivalence of the two notions of \emph{strongly commuting} would also follow from \Cref{C:rigid}, assuming $C(\fn)$ or $C(\fm)$ is rigid.

(3) 
Note that the opposite algebra $\Pi^\circ$ of $\Pi$ is isomorphic to $\Pi$.
The involution of $\Pi$ sending $e_i\mapsto e_{-i-1},f_i\mapsto f_{-i-1}$ gives rise to an autoequivalence of the category of $\Pi$-modules. Composing with $x\mapsto x^*$, where $x^*$ is the dual module, induces the involution $(-)^\lor\colon \co\mapsto \co,\, C(\fm)\mapsto C(\fm^\lor)$. In particular, for $\fn,\fm\in \mq$
\[\hp(C(\fm),C(\fn))=\hp(C(\fn^\lor),C(\fm^\lor)),\] reflecting the last property of \Cref{T:intertwin}.
 Moreover,
\[\hp(C(\fm),C(\fn))=\mathrm{Hom}_{\Pi^\circ}(C(\fm)^*,C(\fn)^*)=\hp(C(\fn^*),C(\fm^*)),\]
where for $C\in\co$ we denote by $C^*$ the irreducible component consisting of the dual modules in $C$. This reflects \Cref{L:aubinv}.

(4) Let $\fm$ be a multisegment such that $C(\fm)$ strongly commutes with itself and assume that $C_{\fm,\fm}$ holds. Then $\Z(\fm)\times Z(\fm)=\pi_+\oplus \pi_-$, where $M_{Z(\fm),\Z(\fm)}$ acts on $\pi_\pm$ via $\pm \lambda,\,\lambda\in\C$. Note that $\lambda\neq 0$ since $M_{Z(\fm),\Z(\fm)}$ is an involution. We fix the sign such that $\Z(\fm+\fm)$ is contained in $\pi_+$. Thus, $C(\fm)$ is in this case rigid if and only if $\pi_-=0$.

(5) As a final remark, we note that the category of $\Pi$-modules is a Calabi-Yau $2$-category, \emph{i.e.} there exists a natural isomorphism
\[\mathrm{Ext}_\Pi^2(x,y)\cong \Hp(y,x)^*,\] \emph{cf.} \cite[§8]{GeisLecSchr07}, \cite[§4.2]{BauKamTin},
for all $\Pi$-modules $x,y$ and the higher Ext-groups vanish. 
\end{rem}
Let us now analyze the following example due to Leclerc, see \cite{Lec03}. We consider all multisegments in $\Ms(\rho)\cong \mq$.
It is the smallest non-rigid multisegment that strongly commutes with itself.
Namely, let \[\fm_{Lec}=[4,5]+[2,4]+[3,3]+[1,2],\, \fm_1=[1,4]+[2,5],\,\fm_2=[1,2]+[2,3]+[3,4]+[4,5].\]
Then \[\Z(\fm_{Lec})\times \Z(\fm_{Lec})=\Z(2\fm_{Lec})\oplus \Z(\fm_1)\times \Z(\fm_2),\]
\[C(\fm_{Lec})*C(\fm_{Lec})=C( 2\fm_{Lec}),\, C(\fm_1)*C(\fm_2)=C(\fm_1\plus\fm_2).\]
Since $M_{Z(\fm_{Lec}),\Z(\fm_{Lec})}$ needs to square to a scalar, we see that up to a scalar the intertwining operator acts on one summand by $1$ and on the other by $-1$. In particular, we have \[\fd(\Z(\fm_{Lec}),\Z(\fm_{Lec}))=0\] and hence \[\Lambda(\Z(\fm_{Lec}),\Z(\fm_{Lec}))=-\frac{1}{2}(\gdim(\fm_{Lec}),\gdim(\fm_{Lec}))=2.\]
On the other hand, the Crawley-Boevey identity implies also $\hp(C(\fm_{Lec}),C(\fm_{Lec})=-\frac{1}{2}(\gdim(\fm_{Lec}),\gdim(\fm_{Lec}))$.

Moreover, we have that for all $\fn\in\mq$, 
\[\Lambda(\Z(\fm_1+\fm_2),\Z(\fn))=\Lambda(\Z(\fm_1),\Z(\fn))+\Lambda(\Z(\fm_2),\Z(\fn))=\]
\[=\hp(C(\fn),C(\fm_1))+\hp(C(\fn),C(\fm_2))=\hp(C(\fn),C(\fm_1+\fm_2)).\]
Note that this last example is of interest, since $\fm_1+\fm_2$ is the first example for which the singular support of the $G_{2\gdim(\fm_{Lec})}$-invariant intersection-cohomology complex on $E^+({2\gdim(\fm_{Lec})})$ corresponding to the orbit parametrized by $\fm_1+\fm_2$ does not equal $C(\fm_1+\fm_2)$, but instead $C(\fm_1+\fm_2)\cup C(\fm_{Lec}+\fm_{Lec})$. For more on this, see for example \cite{KashSai}.
\subsection{Normalized intertwining operators}\label{S:norint}
In applications to the global representation theory of automorphic forms, the normalized intertwining operators 
\[M_{\pi,\sigma}^{nr}(s)\colon \sigma\abs^s\times\pi\abs^{-s}\ra\pi\abs^{-s}\times\sigma\abs^s\] are often used instead of the intertwining operators we considered so far. The normalization factor usually takes the following form.
We fix an additive character $\psi$ of $\Ff$ and recall in the case $\DD=\Ff$ the Rankin-Selberg $L$-factors $L(s,\L(\fm)\times \L(\fn))$ and $\epsilon$-factors $\epsilon(s,\L(\fm),\L(\fn),\psi)$ for any two multisegments $\fm,\fn$, \emph{cf.} \cite{JacPiaSha}. Following \cite[§I]{MoeWal89}, see also \cite{Sha83}, we set
\[r(s,\L(\fm),\L(\fn))\coloneq \frac{L(2s,\L(\fm)^\lor\times\L(\fn))}{L(1+2s,\L(\fm)^\lor\times\L(\fn))\epsilon(2s,\L(\fm)^\lor\times\L(\fn),\psi)}.\]
If $\DD\neq \Ff$ we use the Jacquet-Langlands correspondence of \cite{DeligneKazhdanVigneras1984} and \cite{Badulescu2002} to define $L(s,\L(\fm)\times \L(\fn))$ and $\epsilon(s,\L(\fm)\times \L(\fn))$.
Recall that if $\fm=\De_1+\ldots+\De_k$, $\fn=\Gamma_1+\ldots+\Gamma_l$ then in the case $\DD=\Ff$
\[L(s,\L(\fm)\times \L(\fn))=\prod_{i,j}L(s,\De_i\times \Gamma_j)\]
and similarly for the $\epsilon$-factor.
We recall from  \cite{DeligneKazhdanVigneras1984} and \cite{Badulescu2002} that the Jacquet-Langlands correspondence is a bijection
\[\mathrm{JL}\colon \{\text{essentially discrete series representation of }\gl_n(\DD)\}\ra\]\[\ra  \{\text{essentially discrete series representation of }\gl_{nd}(\Ff)\}\] and that on both sides, the representations are precisely the representations of the form $\L(\De)$ for $\De$ a segment.
We also recall that if $\rho$ is a cuspidal representation of $\gl_n(\DD)$, one has that $\mathrm{JL}(\rho)=\L(\De)$ for some segment $\De=[a_\rho,b_\rho]_{\rho'}$, where $s(\rho)=b_\rho-a_\rho+1$, $\nur=\lvert-\rvert^{s(\rho)}$. Moreover, if $[a,b]_\rho$ is a segment, then \[\mathrm{JL}(\L(\De))=\L([a_\rho+as(\rho),\ldots, b_\rho+bs(\rho)]_{\rho'}).\]
We set then in the case $\DD\neq \Ff$
\[L(s,\L(\fm)\times \L(\fn))\coloneq \prod_{i=1}^k\prod_{j=1}^lL(s\cdot s(\rho_i),\mathrm{JL}(\L(\De_i))\times \mathrm{JL}(\L(\Gamma_j))),\]
where $\De_i\in \Ms(\rho_i)$,
and similarly for the $\epsilon$-factor.
Finally, we set
\[M_{\L(\fm),\L(\fn)}^{nr}(s)\coloneq r(s,\L(\fm),\L(\fn))^{-1}M_{\L(\fm),\L(\fn)}(s).\]
We let $\Lambda^{nr}(\pi,\sigma)$ be order of $M_{\pi,\sigma}^{nr}(s)$ at $0$.
For our considerations we are just interested in the order of the pole of $r(s,\L(\fm),\L(\fn))$ at $s=0$, which we denote by $\alpha_+(\L(\fm),\L(\fn))$.

The following two properties completely determine $\alpha_+$, see \cite[Remark I.2.(5), Eq.(1) I.4]{MoeWal89}.
\begin{enumerate}
    \item Let $\fm,\fn\in \Ms$ and write $\fm=\De_1+\ldots+\De_k,\, \fn=\Gamma_1+\ldots+\Gamma_l$. Then 
    \[\alpha_+(\L(\fm),\L(\fn))=\prod_{i=1}^k\prod_{j=1}^l\alpha_+(\L(\De_i),\L(\Gamma_j)).\]
    \item If $\De,\De'$ are two segments, then
    \[\alpha_+(\L(\De),\L(\De'))=\begin{cases}
        1&\text{ if }\sh{\De'}\prec\De, \De\nprec\De',\\-1&\text{ if }\De'\prec\De,\, \sh{\De}\nprec\De',\\0&\text{otherwise}.
    \end{cases}\]
\end{enumerate}
Note that in the case $\DD\neq\Ff$ the properties follow immediately from the definitions and the above discussed properties.
\begin{corollary}
    Let $\fm,\fn\in \Ms$ with block decomposition $\fm=\fm_1+\ldots+\fm_k, \fn=\fn_1+\ldots+\fn_k$ as in \Cref{C:block}.
    Then \[\alpha_+(\L(\fm),\L(\fn))=\alpha_+(\L(\fm_1),\L(\fn_1))+\ldots+\alpha_+(\L(\fm_k),\L(\fn_k))\]
    and for all $\ain{i}{1}{k}$
    \[\alpha_+(\L(\fm_i),\L(\fn_i))=\alpha_+(\mu_+(\fm_i),\mu_+(\fn_i))\]
\end{corollary}
\begin{proof}
    The claim follows immediately once one computes the following.
    \[\dim_\C\ho_{Q^+}(\mu_+(\De),\mu_+(\Gamma))=\begin{cases}
        1&\text{ if }\sh{\Gamma}\prec\De,\\0&\text{ otherwise.}\\
    \end{cases}\]
      \[\dim_\C\mathrm{Ext}_{Q^+}^1(\mu_+(\De),\mu_+(\Gamma))=\begin{cases}
        1&\text{ if }\De\prec\Gamma,\\0&\text{ otherwise.}\\
    \end{cases}\]
\end{proof}
Recall from \Cref{S:coxeter} that 
\[\dim_\C\ker(\mathcal{T}_{x,y})=\dim_\C\Hp(x,y)\]and\[ \dim_\C\mathrm{coker}(\mathcal{T}_{x,y})=\dim_\C\Hp(x,y)-\alpha_+(x,y).\]
It follows that \Cref{C:int} is equivalent to the following conjecture.
\begin{conjecture}\label{C:intnor}    
    Let $\fm,\fn\in\Ms(\rho)$. Then \[\Lambda^{nr}(\L(\fm),\L(\fn))=\dim_\C\mathrm{coker}(\mathcal{T}_{C(\fm),C(\fn)}).\]
\end{conjecture}
\begin{corollary}\label{C:true}
    Let $\fm,\fn\in\Ms(\rho)$ such that at least one of $\fm,\fn,\fm^*,\fn^*$ is balanced. Then \[\Lambda^{nr}(\L(\fm),\L(\fn))=\dim_\C\mathrm{coker}(\mathcal{T}_{C(\fm),C(\fn)}).\]
\end{corollary}
Recall that a multisegment $\fm=\De_1+\ldots+\De_k\in\Ms(\rho), \,\De_i=[a_i,b_i]_\rho$ is called \emph{Speh} if $\De_{i+1}=\sh{\De_i}$ for $\ain{i}{1}{k-1}$. Note that for such $\fm$, $\fm^*$ is again Speh and simple to compute, namely \[\fm^*=[b_k,b_1]_\rho+\ldots+[a_k,a_1]_\rho.\]
\begin{corollary}
    Let $\pi,\pi'\in\irr$ be two irreducible unitary representations. Then
    \[\Lambda^{nr}(\pi,\pi')=0\] or equivalently
    \[M_{\pi,\pi'}^{nr}(s)\] is holomorphic at $s=0$.
\end{corollary}
\begin{proof}
    Recall that  if $\pi,\pi'$ are irreducible unitary representations, then $\pi\times\pi'$ is irreducible. In the case of $\DD=\Ff$ this a consequence of the results of \cite{Bernstein1984} and for $\DD\neq\Ff$ this follows from \cite{Secherre2009}.
    Hence we can apply \Cref{L:strcom} to obtain that $\fdp(\pi,\pi')=0$. Recall now from the classification of the unitary dual of $G_n$ of \cite{Tadic1986} that since $\pi$ is unitary, one can write it as $\pi_1\times\ldots\times\pi_k$ where $\pi_i\cong \L(\fm_i)$ for some Speh representations $\fm_i$ and similarly for $\pi'$. We thus can apply \Cref{C:true} and \Cref{T:intertwin} to obtain that $0\le \Lambda^{nr}(\pi,\pi')+\Lambda^{nr}(\pi',\pi)=0$ and hence the claim follows.
\end{proof}
\subsection{Best matching functions}
Finally, let us give as a consequence of \Cref{T:main} the following computation of $\Lambda(\Z(\fm),\Z((\fn))$ due to \cite[Remark 2.2]{LapMin25} if at least one of the multisegments is a ladder. Recall a multisegment $[a_1,b_1]+\ldots+[a_k,b_k]$ is called a \emph{ladder} if $a_1>\ldots>a_k,\, b_1>\ldots>b_k$.
Let $\fm=\De_1+\ldots+\De_k, \fn=\Gamma_1+\ldots+\Gamma_j\in\mq$ be two multisegments. We recall from \cite[§5]{LapMin16}
\[X_{\fn,\fm}=\{(i,j)\in \{1,\ldots,k\}\times \{1,\ldots,j\}:\De_i\prec\Ga_j\}\]
and 
\[Y_{\fn,\fm}=\{(i,j)\in \{1,\ldots,k\}\times \{1,\ldots,j\}:\sh{\De_i}\prec\Ga_j\}.\]
We define a relation $\rsa$ between $Y_{\fn,\fm}$ and $X_{\fn,\fm}$ via
\[(r,s)\rsa(r',s')\Leftrightarrow r=r'\text{ and }\Gamma_s\prec\Ga_{s'}\text{ or }s=s'\text{ and }\De_{r'}\prec\De_r.\]
A matching function $(f,I)$ for $(\fn,\fm)$ consists of a subset $I\subseteq X_{\fn,\fm}$ and an injective function $f\colon I\ra Y_{\fn,\fm}$ such that for all $i\in I$, $f(i)\rsa i$. A matching function $(f,I)$ is called a best matching function of $(\fn,\fm)$ if $\# I=\max_{(f',I')}\#I',$ where $(f',I')$ ranges over all matching functions of $(\fn,\fm)$. 
\begin{lemma}[{\cite[6.4]{MinLa18}, \cite[8.3]{LapMin25}}]
    Let $\fm,\fn\in \mq$ and either $\fm$ or $\fn$ be a ladder. Let moreover $(f,I)$ be a best matching function
    of $(\fm,\fn)$. Then
    \[\dim_\C\mathrm{coker}(\mathcal{T}_{C(\fm),C(\fn)})=\#X_{\fm,\fn}\setminus I,\]
    \[\hp(C(\fm),C(\fn))=\#Y_{\fm,\fn}\setminus f(I).\]
\end{lemma}
\begin{corollary}\label{C:bmpole}
Let $\fm,\fn\in \Ms(\rho)\cong\mq$ and either $\fm$ or $\fn$ be a ladder. Let moreover $(f,I)$ be a best matching function
    of $(\fm,\fn)$.
     Then \[\Lambda^{nr}(\L(\fm),\L(\fn))=\#X_{\fm,\fn}\setminus I\] and \[\Lambda(\L(\fm),\L(\fn))=\#Y_{\fm,\fn}\setminus f(I).\]
\end{corollary}
As an example, we will give a formula for the intertwining operator between two Speh-representations. 
\begin{corollary}\label{C:Speh}
    Let $\fm=\De_1+\ldots+\De_k,\fn=\Gamma_1+\ldots+\Gamma_k\in \Ms(\rho)$ be two Speh-multisegments.
    Then \[\Lambda(\L(\fm),\L(\fn))=\min(\#\{j:\sh{\Gamma_j}\prec\De_1\},\#\{i:\sh{\Gamma_l}\prec\De_i\})\]
    and
    \[\Lambda^{nr}(\L(\fm),\L(\fn))=\min(\#\{j:{\Gamma_j}\prec\De_k\},\#\{i:{\Gamma_1}\prec\De_i\}).\]    
\end{corollary}
The author would like to thank Alberto M{\'i}nguez for suggesting the above explicit formula to him.
\begin{proof}
Set $I_1=\{(i,j):{\Gamma_j}\prec\De_i,\, i<k\}$ and $I_2=\{(i,j):{\Gamma_j}\prec\De_i,\, j>1\}$. Then we can consider the matching functions $(f_1, I_1), f_1(i,j)=(i+1,j)$ and $(f_2, I_2), f_2(i,j)=(i,j-1)$. Note that by \Cref{C:bmpole} the following three statements are equivalent: \begin{enumerate}
    \item $\Lambda(\L(\fm),\L(\fn))=\min(\#\{j:\sh{\Gamma_j}\prec\De_1\},\#\{i:\sh{\Gamma_l}\prec\De_i\}),$
    \item $\Lambda^{nr}(\L(\fm),\L(\fn))=\min(\#\{j:{\Gamma_j}\prec\De_k\},\#\{i:{\Gamma_1}\prec\De_i\}),$
    \item Either $(f_1,I_1)$ or $(f_2,I_2)$ is a best matching function.
\end{enumerate} 
In particular, we have that all three claims hold if either $k=1$ or $l=1$ by the greedy algorithm of \cite[§4.6]{LapMin16}. 
For the general case we will prove (1) inductively on $l+k$ via \Cref{L:central}. Indeed, we can use the lemma with $\fa=\De_k^*$, in which case $\fa(\L\fm))=\L(\De_k),\, \Dde(\L(\fm))=\L(\De_1+\ldots+\De_{k-1})$ and $\fa(\L(\fn))$ is non-zero if and only if $\sh{\Gamma_l}\prec \De_k$. In this case write $\De_k=[a,b]_\rho,\, \Gamma_l=[c,d]_\rho$, and it is not hard to see that $\fa(\L(\fn))=\L([a,d]_\rho)$ and $\Dde(\L(\fn))=\L(\Gamma_1+\ldots+\Gamma_{l-1}+[c,a-1]_\rho)$. Finally, using \Cref{L:polevanish}, it is also easy to see that $\Lambda(\Dde(\fm),\Dde(\fn))=\Lambda(\Dde(\fm), \Gamma_1+\ldots+\Gamma_{l-1}).$

The upshot of this discussion is that
\[\Lambda(\L(\fm),\L(\fn))=\begin{cases}
    \Lambda(\L(\fm-\De_k),\L(\fn-\Gamma_l))+1&\text{ if }\sh{\Gamma_l}\prec \De_k,\, \sh{\Gamma_l}\prec \De_1, \\
    \Lambda(\L(\fm-\De_k),\L(\fn-\Gamma_l))&\text{ if }\sh{\Gamma_l}\prec \De_k,\, \sh{\Gamma_l}\nprec \De_1, \\
    \Lambda(\L(\fm-\De_k),\L(\fn))&\text{ if }\sh{\Gamma_l}\nprec \De_k.
\end{cases}\]
Note that, on the other hand, we have the following.
\[\min(\#\{j:\sh{\Gamma_j}\prec\De_1\},\#\{i:\sh{\Gamma_l}\prec\De_i\})=\]\[=\begin{cases}
    \min(\#\{j:\sh{\Gamma_j}\prec\De_1, j<l\},\#\{i:\sh{\Gamma_{l}}\prec\De_i,i>1\})+1&\text{ if }\sh{\Gamma_l}\prec \De_k,\, \sh{\Gamma_l}\prec \De_1,\\
    \min(\#\{j:\sh{\Gamma_j}\prec\De_1, j<l\},\#\{i:\sh{\Gamma_{l}}\prec\De_i,i>1\})&\text{ if }\sh{\Gamma_l}\prec \De_k,\, \sh{\Gamma_l}\nprec \De_1,\\
    \min(\#\{j:\sh{\Gamma_j}\prec\De_1\},\#\{i:\sh{\Gamma_{l}}\prec\De_i,i<k\})&\text{ if }\sh{\Gamma_l}\nprec \De_k.
\end{cases}\]
Finally observe that
\[\#\{i:\sh{\Gamma_{l}}\prec\De_i,i>1\})=\#\{i:\sh{\Gamma}_{l-1}\prec\De_i, i<k\}).\]
The induction step follows now easily.
\end{proof}
\bibliographystyle{abbrv}
\bibliography{Reference.bib}

@article{Ber,
AUTHOR = {Bernstein, I. N. and Zelevinsky, A. V.},
     TITLE = {Induced representations of reductive {${\mathfrak{ p}}$}-adic
              groups. {I}},
   JOURNAL = {Ann. Sci. \'{E}cole Norm. Sup. (4)},
  FJOURNAL = {Annales Scientifiques de l'\'{E}cole Normale Sup\'{e}rieure.
              Quatri\`eme S\'{e}rie},
    VOLUME = {10},
      YEAR = {1977},
    NUMBER = {4},
     PAGES = {441--472},
      ISSN = {0012-9593},
   MRCLASS = {22E50},
  MRNUMBER = {579172},
       URL = {http://www.numdam.org/item?id=ASENS_1977_4_10_4_441_0},
}

@incollection{DeligneKazhdanVigneras1984,
  author    = {Deligne, Pierre and Kazhdan, David and Vign{\'e}ras, Marie-France},
  title     = {Repr{\'e}sentations des alg{\`e}bres centrales simples $p$-adiques},
  booktitle = {Repr{\'e}sentations des Groupes R{\'e}ductifs sur un Corps Local},
  series    = {Travaux en Cours},
  volume    = {8},
  publisher = {Hermann},
  address   = {Paris},
  year      = {1984},
  pages     = {33--117},
  isbn      = {2-7056-5989-7}
}

@article{Badulescu2002,
  author    = {Badulescu, Alexandru Ioan},
  title     = {Correspondance de {J}acquet-{L}anglands pour les corps locaux de caract{\'e}ristique non nulle},
  journal   = {Annales Scientifiques de l'{\'E}cole Normale Sup{\'e}rieure},
  volume    = {35},
  number    = {5},
  year      = {2002},
  pages     = {695--747},
  doi       = {10.1016/S0012-9593(02)01106-0}
}

@article{Tadic1986,
  author    = {Tadi{\'c}, Marko},
  title     = {Unitary representations of general linear group},
  journal   = {Annales scientifiques de l'{\'E}cole Normale Sup{\'e}rieure},
  series    = {4e s{\'e}rie},
  volume    = {19},
  number    = {3},
  pages     = {335--382},
  year      = {1986},
  publisher = {Soci{\'e}t{\'e} Math{\'e}matique de France},
  doi       = {10.24033/asens.1508}
}

@article{Secherre2009,
  author    = {S{\'e}cherre, Vincent},
  title     = {Proof of the {T}adi{\'c} conjecture {$(U_0)$} on the unitary dual of {$\mathrm{GL}_m(D)$}},
  journal   = {Journal f{\"u}r die reine und angewandte Mathematik (Crelles Journal)},
  volume    = {2009},
  number    = {633},
  pages     = {217--243},
  year      = {2009},
  doi       = {10.1515/crelle.2009.007}
}

@incollection{Bernstein1984,
  author    = {Bernstein, Joseph N.},
  title     = {$P$-invariant distributions on $\mathrm{GL}(N)$ and the classification of unitary representations of $\mathrm{GL}(N)$ (non-{A}rchimedean case)},
  booktitle = {Lie Group Representations {II}},
  series    = {Lecture Notes in Mathematics},
  volume    = {1041},
  editor    = {Herb, Rebecca and Kudla, Stephen and Lipsman, Ronald and Rosenberg, Jonathan},
  publisher = {Springer},
  address   = {Berlin, Heidelberg},
  pages     = {50--102},
  year      = {1984},
  doi       = {10.1007/BFb0073145}
}

@article{MinguezSecherreStevens2014,
  author  = {M{\'{\i}}nguez, Alberto and S{\'e}cherre, Vincent and Stevens, Shaun},
  title   = {Types modulo $\ell$ pour les formes int{\'e}rieures de {$\mathrm{GL}_n$} sur un corps local non archim{\'e}dien},
  journal = {Proceedings of the London Mathematical Society},
  volume  = {109},
  number  = {4},
  pages   = {823--891},
  year    = {2014},
  doi     = {10.1112/plms/pdu020}
}

@article {Cra,
    AUTHOR = {Crawley-Boevey, William},
     TITLE = {On the exceptional fibres of {K}leinian singularities},
   JOURNAL = {Amer. J. Math.},
  FJOURNAL = {American Journal of Mathematics},
    VOLUME = {122},
      YEAR = {2000},
    NUMBER = {5},
     PAGES = {1027--1037},
      ISSN = {0002-9327,1080-6377},
   MRCLASS = {14B05 (16G20)},
  MRNUMBER = {1781930},
MRREVIEWER = {Michel\ Van den Bergh},
       URL =
              {http://muse.jhu.edu/journals/american_journal_of_mathematics/v122/122.5crawley-boevey.pdf},
}

@article {KashSai,
    AUTHOR = {Kashiwara, Masaki and Saito, Yoshihisa},
     TITLE = {Geometric construction of crystal bases},
   JOURNAL = {Duke Math. J.},
  FJOURNAL = {Duke Mathematical Journal},
    VOLUME = {89},
      YEAR = {1997},
    NUMBER = {1},
     PAGES = {9--36},
      ISSN = {0012-7094,1547-7398},
   MRCLASS = {17B37},
  MRNUMBER = {1458969},
MRREVIEWER = {Kailash\ C.\ Misra},
       DOI = {10.1215/S0012-7094-97-08902-X},
       URL = {https://doi.org/10.1215/S0012-7094-97-08902-X},
}

@article {GeiSchr,
    AUTHOR = {Geiss, Christof and Schr\"oer, Jan},
     TITLE = {Extension-orthogonal components of preprojective varieties},
   JOURNAL = {Trans. Amer. Math. Soc.},
  FJOURNAL = {Transactions of the American Mathematical Society},
    VOLUME = {357},
      YEAR = {2005},
    NUMBER = {5},
     PAGES = {1953--1962},
      ISSN = {0002-9947,1088-6850},
   MRCLASS = {16G10 (14M15 16G20)},
  MRNUMBER = {2115084},
MRREVIEWER = {Sverre\ O.\ Smal\o},
       DOI = {10.1090/S0002-9947-04-03555-X},
       URL = {https://doi.org/10.1090/S0002-9947-04-03555-X},
}

@article{ZelIII,
  title={$p$-adic analog of the {K}azhdan-{L}usztig hypothesis},
  author={Zelevinsky, A. V. },
  journal={Functional Analysis and Its Applications},
  year={1981},
  volume={15},
  pages={83-92}
}

@article {Lec03,
    AUTHOR = {Leclerc, B.},
     TITLE = {Imaginary vectors in the dual canonical basis of {$U_q(\mathfrak{
              n})$}},
   JOURNAL = {Transform. Groups},
  FJOURNAL = {Transformation Groups},
    VOLUME = {8},
      YEAR = {2003},
    NUMBER = {1},
     PAGES = {95--104},
      ISSN = {1083-4362,1531-586X},
   MRCLASS = {17B37},
  MRNUMBER = {1959765},
MRREVIEWER = {Dongho\ Moon},
       DOI = {10.1007/BF03326301},
       URL = {https://doi.org/10.1007/BF03326301},
}

@article{SchneiStu,
issn = {0073-8301},
journal = {Publications mathématiques. Institut des hautes études scientifiques},
pages = {97--191},
volume = {85},
publisher = {Springer},
number = {1},
year = {1997},
title = {Representation theory and sheaves on the {B}ruhat-{T}its building},
copyright = {Copyright 2010 Elsevier B.V., All rights reserved.},
language = {eng},
address = {Heidelberg},
author = {Schneider, Peter and Stuhler, Ulrich},
}

@ARTICLE{BerBerKaz,
	author = {Bernstein, Joseph and Bezrukavnikov, Roman and Kazhdan, David},
	title = {Deligne–{L}usztig duality and wonderful compactification},
	year = {2018},
	journal = {Selecta Mathematica, New Series},
	volume = {24},
	number = {1},
	pages = {7 – 20},
	doi = {10.1007/s00029-018-0391-5},
	url = {https://www.scopus.com/inward/record.uri?eid=2-s2.0-85040863311&doi=10.1007%2fs00029-018-0391-5&partnerID=40&md5=bf5fbc8c23d167bfa03024b90887cddc},
	type = {Article},
	publication_stage = {Final},
	source = {Scopus}
}

@article {Aub,
    AUTHOR = {Aubert, Anne-Marie},
     TITLE = {Dualit\'{e} dans le groupe de {G}rothendieck de la
              cat\'{e}gorie des repr\'{e}sentations lisses de longueur finie
              d'un groupe r\'{e}ductif {$p$}-adique},
   JOURNAL = {Trans. Amer. Math. Soc.},
  FJOURNAL = {Transactions of the American Mathematical Society},
    VOLUME = {347},
      YEAR = {1995},
    NUMBER = {6},
     PAGES = {2179--2189},
      ISSN = {0002-9947,1088-6850},
   MRCLASS = {22E50 (20G05 20G25 20G40)},
  MRNUMBER = {1285969},
MRREVIEWER = {Ernst-Wilhelm\ Zink},
       DOI = {10.2307/2154931},
       URL = {https://doi.org/10.2307/2154931},
}

@article {MinLa18,
    AUTHOR = {Lapid, E. and M\'{\i}nguez, A.},
     TITLE = {Geometric conditions for {$\square$}-irreducibility of certain
              representations of the general linear group over a
              non-archimedean local field},
   JOURNAL = {Adv. Math.},
  FJOURNAL = {Advances in Mathematics},
    VOLUME = {339},
      YEAR = {2018},
     PAGES = {113--190},
      ISSN = {0001-8708,1090-2082},
   MRCLASS = {20G25 (20C33)},
  MRNUMBER = {3866895},
MRREVIEWER = {Daniel\ Goldstein},
       DOI = {10.1016/j.aim.2018.09.027},
       URL = {https://doi.org/10.1016/j.aim.2018.09.027},
}

@article {BerZel76,
    AUTHOR = {Bernstein, I. N. and Zelevinsky, A. V.},
     TITLE = {Representations of the group {$GL(n,F),$} where {$F$} is a
              local non-{A}rchimedean field},
   JOURNAL = {Uspehi Mat. Nauk},
  FJOURNAL = {Akademija Nauk SSSR i Moskovskoe Matemati\v{c}eskoe
              Ob\v{s}\v{c}estvo. Uspehi Matemati\v{c}eskih Nauk},
    VOLUME = {31},
      YEAR = {1976},
    NUMBER = {3(189)},
     PAGES = {5--70},
      ISSN = {0042-1316},
   MRCLASS = {22E50},
  MRNUMBER = {425030},
MRREVIEWER = {G.\ I.\ Ol\cprime shanski\u{\i}},
}

@article {LapMin25,
    AUTHOR = {Lapid, Erez and M\'inguez, Alberto},
     TITLE = {A binary operation on irreducible components of {L}usztig's
              nilpotent varieties {II}: applications and conjectures for
              representations of {${\rm GL}_n$} over a non-archimedean local
              field},
   JOURNAL = {Pure Appl. Math. Q.},
  FJOURNAL = {Pure and Applied Mathematics Quarterly},
    VOLUME = {21},
      YEAR = {2025},
    NUMBER = {2},
     PAGES = {813--863},
      ISSN = {1558-8599,1558-8602},
   MRCLASS = {22E50 (14M15 16G20 16T20 20G42)},
  MRNUMBER = {4847251},
       DOI = {10.4310/pamq.241205005734},
       URL = {https://doi.org/10.4310/pamq.241205005734},
}

@article {Dat,
    AUTHOR = {Dat, J.-F.},
     TITLE = {{$v$}-tempered representations of {$p$}-adic groups. {I}.
              {$l$}-adic case},
   JOURNAL = {Duke Math. J.},
  FJOURNAL = {Duke Mathematical Journal},
    VOLUME = {126},
      YEAR = {2005},
    NUMBER = {3},
     PAGES = {397--469},
      ISSN = {0012-7094,1547-7398},
   MRCLASS = {22E50 (11F70)},
  MRNUMBER = {2120114},
MRREVIEWER = {Anton\ Deitmar},
       DOI = {10.1215/S0012-7094-04-12631-4},
       URL = {https://doi.org/10.1215/S0012-7094-04-12631-4},
}

@article {Zel,
    AUTHOR = {Zelevinsky, A. V.},
     TITLE = {Induced representations of reductive {${\mathfrak{ p}}$}-adic
              groups. {II}. {O}n irreducible representations of {${\rm
              GL}(n)$}},
   JOURNAL = {Ann. Sci. \'{E}cole Norm. Sup. (4)},
  FJOURNAL = {Annales Scientifiques de l'\'{E}cole Normale Sup\'{e}rieure.
              Quatri\`eme S\'{e}rie},
    VOLUME = {13},
      YEAR = {1980},
    NUMBER = {2},
     PAGES = {165--210},
      ISSN = {0012-9593},
   MRCLASS = {22E50},
  MRNUMBER = {584084},
       URL = {http://www.numdam.org/item?id=ASENS_1980_4_13_2_165_0},
}

@article {Wal03,
    AUTHOR = {Waldspurger, J.-L.},
     TITLE = {La formule de {P}lancherel pour les groupes {$p$}-adiques
              (d'apr\`es {H}arish-{C}handra)},
   JOURNAL = {J. Inst. Math. Jussieu},
  FJOURNAL = {Journal of the Institute of Mathematics of Jussieu. JIMJ.
              Journal de l'Institut de Math\'ematiques de Jussieu},
    VOLUME = {2},
      YEAR = {2003},
    NUMBER = {2},
     PAGES = {235--333},
      ISSN = {1474-7480,1475-3030},
   MRCLASS = {22E35 (22E50)},
  MRNUMBER = {1989693},
MRREVIEWER = {Rebecca\ Herb},
       DOI = {10.1017/S1474748003000082},
       URL = {https://doi.org/10.1017/S1474748003000082},
}

@article {KKKOhead,
    AUTHOR = {Kang, Seok-Jin and Kashiwara, Masaki and Kim, Myungho and Oh,
              Se-jin},
     TITLE = {Simplicity of heads and socles of tensor products},
   JOURNAL = {Compos. Math.},
  FJOURNAL = {Compositio Mathematica},
    VOLUME = {151},
      YEAR = {2015},
    NUMBER = {2},
     PAGES = {377--396},
      ISSN = {0010-437X,1570-5846},
   MRCLASS = {17B67 (17B10 17B37)},
  MRNUMBER = {3314831},
MRREVIEWER = {Rupert\ Wei Tze Yu},
       DOI = {10.1112/S0010437X14007799},
       URL = {https://doi.org/10.1112/S0010437X14007799},
}

@article {KKKO,
    AUTHOR = {Kang, Seok-Jin and Kashiwara, Masaki and Kim, Myungho and Oh,
              Se-jin},
     TITLE = {Monoidal categorification of cluster algebras},
   JOURNAL = {J. Amer. Math. Soc.},
  FJOURNAL = {Journal of the American Mathematical Society},
    VOLUME = {31},
      YEAR = {2018},
    NUMBER = {2},
     PAGES = {349--426},
      ISSN = {0894-0347,1088-6834},
   MRCLASS = {13F60 (16Gxx 17B37 18D10 81R50)},
  MRNUMBER = {3758148},
MRREVIEWER = {Fan\ Qin},
       DOI = {10.1090/jams/895},
       URL = {https://doi.org/10.1090/jams/895},
}

@article {GeisLecSchr,
    AUTHOR = {Gei\ss, Christof and Leclerc, Bernard and Schr\"oer, Jan},
     TITLE = {Rigid modules over preprojective algebras},
   JOURNAL = {Invent. Math.},
  FJOURNAL = {Inventiones Mathematicae},
    VOLUME = {165},
      YEAR = {2006},
    NUMBER = {3},
     PAGES = {589--632},
      ISSN = {0020-9910,1432-1297},
   MRCLASS = {16G20 (16D70 16G70 17B37)},
  MRNUMBER = {2242628},
MRREVIEWER = {Csaba\ Sz\'ant\'o},
       DOI = {10.1007/s00222-006-0507-y},
       URL = {https://doi.org/10.1007/s00222-006-0507-y},
}

@article {BauKamTin,
    AUTHOR = {Baumann, Pierre and Kamnitzer, Joel and Tingley, Peter},
     TITLE = {Affine {M}irkovi\'c-{V}ilonen polytopes},
   JOURNAL = {Publ. Math. Inst. Hautes \'Etudes Sci.},
  FJOURNAL = {Publications Math\'ematiques. Institut de Hautes \'Etudes
              Scientifiques},
    VOLUME = {120},
      YEAR = {2014},
     PAGES = {113--205},
      ISSN = {0073-8301,1618-1913},
   MRCLASS = {17B37 (16G20 17B67)},
  MRNUMBER = {3270589},
MRREVIEWER = {Csaba\ Sz\'ant\'o},
       DOI = {10.1007/s10240-013-0057-y},
       URL = {https://doi.org/10.1007/s10240-013-0057-y},
}

@article {GeisLecSchr07,
    AUTHOR = {Geiss, Christof and Leclerc, Bernard and Schr\"oer, Jan},
     TITLE = {Semicanonical bases and preprojective algebras. {II}. {A}
              multiplication formula},
   JOURNAL = {Compos. Math.},
  FJOURNAL = {Compositio Mathematica},
    VOLUME = {143},
      YEAR = {2007},
    NUMBER = {5},
     PAGES = {1313--1334},
      ISSN = {0010-437X,1570-5846},
   MRCLASS = {17B37 (17B35 17B67)},
  MRNUMBER = {2360317},
MRREVIEWER = {Csaba\ Sz\'ant\'o},
       DOI = {10.1112/S0010437X07002977},
       URL = {https://doi.org/10.1112/S0010437X07002977},
}

@article {HerLec,
    AUTHOR = {Hernandez, David and Leclerc, Bernard},
     TITLE = {Cluster algebras and quantum affine algebras},
   JOURNAL = {Duke Math. J.},
  FJOURNAL = {Duke Mathematical Journal},
    VOLUME = {154},
      YEAR = {2010},
    NUMBER = {2},
     PAGES = {265--341},
      ISSN = {0012-7094,1547-7398},
   MRCLASS = {17B37 (13F60)},
  MRNUMBER = {2682185},
MRREVIEWER = {Jan\ E.\ Grabowski},
       DOI = {10.1215/00127094-2010-040},
       URL = {https://doi.org/10.1215/00127094-2010-040},
}

@article {Jader,
AUTHOR = {Jantzen, Chris},
     TITLE = {Jacquet modules of {$p$}-adic general linear groups},
   JOURNAL = {Represent. Theory},
  FJOURNAL = {Representation Theory. An Electronic Journal of the American
              Mathematical Society},
    VOLUME = {11},
      YEAR = {2007},
     PAGES = {45--83},
      ISSN = {1088-4165},
   MRCLASS = {22E50},
  MRNUMBER = {2306606},
MRREVIEWER = {Anne-Marie\ H.\ Aubert},
       DOI = {10.1090/S1088-4165-07-00316-0},
       URL = {https://doi.org/10.1090/S1088-4165-07-00316-0},
}

@article {Mder,
    AUTHOR = {M\'{\i}nguez, Alberto},
     TITLE = {Sur l'irr\'{e}ductibilit\'{e} d'une induite parabolique},
   JOURNAL = {J. Reine Angew. Math.},
  FJOURNAL = {Journal f\"{u}r die Reine und Angewandte Mathematik. [Crelle's
              Journal]},
    VOLUME = {629},
      YEAR = {2009},
     PAGES = {107--131},
      ISSN = {0075-4102,1435-5345},
   MRCLASS = {22E50},
  MRNUMBER = {2527415},
MRREVIEWER = {Shaun\ A. R. Stevens},
       DOI = {10.1515/CRELLE.2009.028},
       URL = {https://doi.org/10.1515/CRELLE.2009.028},
}

@article {JacPiaSha,
    AUTHOR = {Jacquet, H. and Piatetskii-Shapiro, I. I. and Shalika, J. A.},
     TITLE = {Rankin-{S}elberg convolutions},
   JOURNAL = {Amer. J. Math.},
  FJOURNAL = {American Journal of Mathematics},
    VOLUME = {105},
      YEAR = {1983},
    NUMBER = {2},
     PAGES = {367--464},
      ISSN = {0002-9327,1080-6377},
   MRCLASS = {11F67 (11F70 11R39 22E55)},
  MRNUMBER = {701565},
MRREVIEWER = {Freydoon\ Shahidi},
       DOI = {10.2307/2374264},
       URL = {https://doi.org/10.2307/2374264},
}

@article {LapMin16,
    AUTHOR = {Lapid, Erez and M\'inguez, Alberto},
     TITLE = {On parabolic induction on inner forms of the general linear
              group over a non-archimedean local field},
   JOURNAL = {Selecta Math. (N.S.)},
  FJOURNAL = {Selecta Mathematica. New Series},
    VOLUME = {22},
      YEAR = {2016},
    NUMBER = {4},
     PAGES = {2347--2400},
      ISSN = {1022-1824,1420-9020},
   MRCLASS = {22E50},
  MRNUMBER = {3573961},
MRREVIEWER = {Dani\ Szpruch},
       DOI = {10.1007/s00029-016-0281-7},
       URL = {https://doi.org/10.1007/s00029-016-0281-7},
}

@article {BezKaz15,
    AUTHOR = {Bezrukavnikov, Roman and Kazhdan, David},
     TITLE = {Geometry of second adjointness for {$p$}-adic groups},
      NOTE = {With an appendix by Yakov Varshavsky, Bezrukavnikov and
              Kazhdan},
   JOURNAL = {Represent. Theory},
  FJOURNAL = {Representation Theory. An Electronic Journal of the American
              Mathematical Society},
    VOLUME = {19},
      YEAR = {2015},
     PAGES = {299--332},
      ISSN = {1088-4165},
   MRCLASS = {20G15 (14G20 14L15 14L24 20G25 22E35 22E50)},
  MRNUMBER = {3430373},
MRREVIEWER = {Marko\ Tadi\'c},
       DOI = {10.1090/ert/471},
       URL = {https://doi.org/10.1090/ert/471},
}

@book {ChrGinz,
    AUTHOR = {Chriss, Neil and Ginzburg, Victor},
     TITLE = {Representation theory and complex geometry},
 PUBLISHER = {Birkh\"auser Boston, Inc., Boston, MA},
      YEAR = {1997},
     PAGES = {x+495},
      ISBN = {0-8176-3792-3},
   MRCLASS = {22E47 (14F99 19L47 20G05 22-02 58F05)},
  MRNUMBER = {1433132},
MRREVIEWER = {William\ M.\ McGovern},
}

@article {KazLus87,
    AUTHOR = {Kazhdan, David and Lusztig, George},
     TITLE = {Proof of the {D}eligne-{L}anglands conjecture for {H}ecke
              algebras},
   JOURNAL = {Invent. Math.},
  FJOURNAL = {Inventiones Mathematicae},
    VOLUME = {87},
      YEAR = {1987},
    NUMBER = {1},
     PAGES = {153--215},
      ISSN = {0020-9910,1432-1297},
   MRCLASS = {11S37 (18F25 19K33 22E50)},
  MRNUMBER = {862716},
MRREVIEWER = {Joe\ Repka},
       DOI = {10.1007/BF01389157},
       URL = {https://doi.org/10.1007/BF01389157},
}

@article {Bor76,
    AUTHOR = {Borel, Armand},
     TITLE = {Admissible representations of a semi-simple group over a local
              field with vectors fixed under an {I}wahori subgroup},
   JOURNAL = {Invent. Math.},
  FJOURNAL = {Inventiones Mathematicae},
    VOLUME = {35},
      YEAR = {1976},
     PAGES = {233--259},
      ISSN = {0020-9910,1432-1297},
   MRCLASS = {22E50 (20G05)},
  MRNUMBER = {444849},
MRREVIEWER = {Allan\ J.\ Silberger},
       DOI = {10.1007/BF01390139},
       URL = {https://doi.org/10.1007/BF01390139},
}

@article {Sha83,
    AUTHOR = {Shahidi, Freydoon},
     TITLE = {Local coefficients and normalization of intertwining operators
              for {${\rm GL}(n)$}},
   JOURNAL = {Compositio Math.},
  FJOURNAL = {Compositio Mathematica},
    VOLUME = {48},
      YEAR = {1983},
    NUMBER = {3},
     PAGES = {271--295},
      ISSN = {0010-437X,1570-5846},
   MRCLASS = {22E50 (11F70 22E55)},
  MRNUMBER = {700741},
MRREVIEWER = {Stephen\ Gelbart},
       URL = {http://www.numdam.org/item?id=CM_1983__48_3_271_0},
}

@article {Lus91,
    AUTHOR = {Lusztig, G.},
     TITLE = {Quivers, perverse sheaves, and quantized enveloping algebras},
   JOURNAL = {J. Amer. Math. Soc.},
  FJOURNAL = {Journal of the American Mathematical Society},
    VOLUME = {4},
      YEAR = {1991},
    NUMBER = {2},
     PAGES = {365--421},
      ISSN = {0894-0347,1088-6834},
   MRCLASS = {17B37 (17B67 20G05)},
  MRNUMBER = {1088333},
MRREVIEWER = {H.\ H.\ Andersen},
       DOI = {10.2307/2939279},
       URL = {https://doi.org/10.2307/2939279},
}

@article {Lus90,
    AUTHOR = {Lusztig, G.},
     TITLE = {Canonical bases arising from quantized enveloping algebras},
   JOURNAL = {J. Amer. Math. Soc.},
  FJOURNAL = {Journal of the American Mathematical Society},
    VOLUME = {3},
      YEAR = {1990},
    NUMBER = {2},
     PAGES = {447--498},
      ISSN = {0894-0347,1088-6834},
   MRCLASS = {17B35 (16A64)},
  MRNUMBER = {1035415},
MRREVIEWER = {Ya.\ S.\ So\u ibel\cprime man},
       DOI = {10.2307/1990961},
       URL = {https://doi.org/10.2307/1990961},
}

@article {KasMisOkYa,
    AUTHOR = {Kashiwara, M. and Misra, K. C. and Okado, M. and Yamada, D.},
     TITLE = {Perfect crystals for {$U_q(D^{(3)}_4)$}},
   JOURNAL = {J. Algebra},
  FJOURNAL = {Journal of Algebra},
    VOLUME = {317},
      YEAR = {2007},
    NUMBER = {1},
     PAGES = {392--423},
      ISSN = {0021-8693,1090-266X},
   MRCLASS = {17B37 (17B67)},
  MRNUMBER = {2360156},
MRREVIEWER = {Dijana\ Jakeli\'c},
       DOI = {10.1016/j.jalgebra.2007.02.021},
       URL = {https://doi.org/10.1016/j.jalgebra.2007.02.021},
}

@misc{OhScr,
      title={Simplicity of tensor products of Kirillov--Reshetikhin modules: nonexceptional affine and G types}, 
      author={Sejin Oh and Travis Scrimshaw},
      year={2020},
      eprint={1910.10347},
      note={arXiv:1910.10347},
      primaryClass={math.QA},
      url={https://arxiv.org/abs/1910.10347}, 
}

@article {FujKot,
    AUTHOR = {Fujita, Ryo and Murakami, Kota},
     TITLE = {Deformed {C}artan matrices and generalized preprojective
              algebras {I}: {F}inite type},
   JOURNAL = {Int. Math. Res. Not. IMRN},
  FJOURNAL = {International Mathematics Research Notices. IMRN},
      YEAR = {2023},
    volume = {8},
     PAGES = {6924--6975},
      ISSN = {1073-7928,1687-0247},
   MRCLASS = {17B37 (16T20 17B67 81R50)},
  MRNUMBER = {4574392},
MRREVIEWER = {Cristian\ Vay},
       DOI = {10.1093/imrn/rnac054},
       URL = {https://doi.org/10.1093/imrn/rnac054},
}

@article {MoeWal89,
    AUTHOR = {M{\oe}glin, C. and Waldspurger, J.-L.},
     TITLE = {Le spectre r\'esiduel de {${\rm GL}(n)$}},
   JOURNAL = {Ann. Sci. \'Ecole Norm. Sup. (4)},
  FJOURNAL = {Annales Scientifiques de l'\'Ecole Normale Sup\'erieure.
              Quatri\`eme S\'erie},
    VOLUME = {22},
      YEAR = {1989},
    NUMBER = {4},
     PAGES = {605--674},
      ISSN = {0012-9593},
   MRCLASS = {22E55 (11F70 22E50)},
  MRNUMBER = {1026752},
MRREVIEWER = {Stephen\ Gelbart},
       URL = {http://www.numdam.org/item?id=ASENS_1989_4_22_4_605_0},
}

@article {MoeWal,
    AUTHOR = {M{\oe}glin, C. and Waldspurger, J.-L.},
     TITLE = {Sur l'involution de {Z}elevinski},
   JOURNAL = {J. Reine Angew. Math.},
  FJOURNAL = {Journal f\"ur die Reine und Angewandte Mathematik. [Crelle's
              Journal]},
    VOLUME = {372},
      YEAR = {1986},
     PAGES = {136--177},
      ISSN = {0075-4102,1435-5345},
   MRCLASS = {22E50},
  MRNUMBER = {863522},
MRREVIEWER = {David\ Manderscheid},
       DOI = {10.1515/crll.1986.372.136},
       URL = {https://doi.org/10.1515/crll.1986.372.136},
}

@article{AizLap,
  author       = {Aizenbud, Avraham and Lapid, Erez},
  title        = {A binary operation on irreducible components of {L}usztig’s nilpotent varieties {I}: definition and properties},
  journal      = {Pure and Applied Mathematics Quarterly},
  year         = {2024},
  volume       = {21},
  number       = {1},
  pages        = {5--41},
  doi          = {10.4310/pamq.241203024730}
}
\end{document}